\documentclass[11pt,english]{smfart}

\usepackage[OT2,T1]{fontenc}
\usepackage{lmodern}
\usepackage{textcomp}
\usepackage[english]{babel}
\usepackage{url,xspace,smfthm}
\usepackage{amsmath,amssymb,amsfonts}
\usepackage{mathtools}
\usepackage{tikz}
\usepackage{tikz-cd}

\usepackage{enumitem}
\usepackage{fancyhdr}
\usepackage{float}

\usepackage[margin=1in]{geometry}
\usepackage{hyperref}
\usepackage{lipsum}
\usepackage[utf8]{inputenc}
\usepackage[T1]{fontenc}
\usepackage{lmodern}

\usepackage{cite} % 引用管理パッケージ

\tikzcdset{arrow style=tikz,
           diagrams={>=Straight Barb}
           }

\DeclareSymbolFont{cyrletters}{OT2}{wncyr}{m}{n}
\DeclareMathSymbol{\Sha}{\mathalpha}{cyrletters}{"58}
\DeclareMathSymbol{\Zhe}{\mathalpha}{cyrletters}{"11}

\newtheoremstyle{uprightstyle}% Name
  {3pt}% Space above
  {3pt}% Space below
  {\normalfont}% Body font
  {}% Indent amount
  {\bfseries}% Theorem head font
  {.}% Punctuation after theorem head
  {5pt plus 1pt minus 1pt}% Space after theorem head
  {}% Theorem head spec (can be left empty, meaning `normal`)

\theoremstyle{remark}

\newtheorem*{unnumberedquestion}{Question}

\theoremstyle{plain} % または definition, remark など

\usetikzlibrary{cd,bbox,babel}
\usepackage{amsmath,amssymb,amsfonts,amscd}

\title{Behaviors of the Tate--Shafarevich Group of Elliptic Curves under Quadratic Field Extensions}

\address{\parbox{\linewidth}{Mathematical Institute, Graduate School of Science, \\
Tohoku University, 6-3 Aramakiaza, Aoba, Sendai, Miyagi 980-8578, Japan.}}

\email{otheiio323.com@gmail.com}


\title{Behaviors of the Tate--Shafarevich Group of Elliptic Curves under Quadratic Field Extensions}

\author{Asuka Shiga}


\address{Mathematical Institute, Graduate School of Science, \\
Tohoku University, 6-3 Aramakiaza, Aoba, Sendai, Miyagi 980-8578, Japan.}

\email{otheiio323.com@gmail.com}

\usepackage{mathtools}
\usepackage{tikz}
\usepackage{tikz-cd}

\usepackage{enumitem}
\usepackage{fancyhdr}
\usepackage{float}

\usepackage{hyperref}
\usepackage{lipsum}
\usepackage[utf8]{inputenc}

\usepackage{lmodern}

\usepackage{cite} % 引用管理パッケージ

\tikzcdset{arrow style=tikz,
           diagrams={>=Straight Barb}
           }

\DeclareSymbolFont{cyrletters}{OT2}{wncyr}{m}{n}
\DeclareMathSymbol{\Sha}{\mathalpha}{cyrletters}{"58}
\DeclareMathSymbol{\Zhe}{\mathalpha}{cyrletters}{"11}

\makeatletter
\newtheorem{theorem}{Theorem}[section]
\newtheorem{definition}[theorem]{Definition}
\newtheorem{lemma}[theorem]{Lemma}
\newtheorem{cor}[theorem]{Corollary}
\newtheorem{remark}[theorem]{Remark}
\newtheorem{example}[theorem]{Example}
\newtheorem{proposition}[theorem]{Proposition}

\@addtoreset{theorem}{section}
\makeatother

\usetikzlibrary{cd,bbox,babel}
\usepackage{amsmath,amssymb,amsfonts,amscd}

\begin{document}

\title{Behaviors of the Tate--Shafarevich group of elliptic curves under quadratic field extensions}

\begin{abstract}

Let $E/\mathbb{Q}$ be an elliptic curve. We study the behavior of the Tate--Shafarevich group of $E$ under quadratic extensions $\mathbb{Q}(\sqrt{D})/\mathbb{Q}$. By analyzing the cokernel of the restriction map, without assuming the finiteness of the Tate--Shafarevich group, we prove that the ratio $\frac{\#\Sha(E/\mathbb{Q}(\sqrt{D}))[4]}{\#\Sha(E_D/\mathbb{Q})[2]}$ and $\#\Sha(E_D/\mathbb{Q})[2]$ can, under some conditions on $E/\mathbb{Q}$, grow arbitrarily large simultaneously, where $E_D$ denotes the quadratic twist of $E$ by $D$. For elliptic curves of the form $E : y^2 = x^3 + px$ with $p\equiv 1 \bmod 4$ being an odd prime, assuming the finiteness of the relevant Tate--Shafarevich groups, we prove that $\#\Sha(E/\mathbb{Q}(\sqrt{D}))[2] \leq 4$ and $\Sha(E_D/\mathbb{Q})[2] = 0$ for infinitely many square-free integers $D$ with $-D$ being a prime number.  Additionally, $\Sha(E/\mathbb{Q}(\sqrt{-D}))[2]\neq 0$ for all $D$ when $p=257$.
\end{abstract}

\maketitle

\tableofcontents

\section{Introduction}
Let $K$ be a number field and $M_K$ be the set of places of $K$. Let $E$ be an elliptic curve over $K$. The Tate--Shafarevich group of $E/K$ is defined as follows:

\[\begin{array}{ccc}
\Sha(E/K) \stackrel{\mathrm{def}}{=} \mathrm{Ker}\bigg(H^1(G_K,E) & \stackrel{\bigoplus_{v}\mathrm{res}_v}{\longrightarrow} & \bigoplus_{v\in M_K} H^1(G_{K_v},E)\bigg)\\

\hspace{2.7cm}\rotatebox{90}{$\in$} & &\rotatebox{90}{$\in$} \\

\hspace{2.7cm} \left[f\right] & \longmapsto & (\left[ f\mid_{G_{K_v}}\right])_v

\end{array}\]  where $K_v$ is the completion of $K$ at the place $v$, and $G_K, G_{K_v}$ are the absolute Galois groups of $K, K_v$ respectively and $\mathrm{res}_v : H^1(G_K,E)\to H^1(G_{K_v},E)$ is the restriction map of Galois cohomology. The Tate--Shafarevich group lives in the global Galois cohomology $H^1(G_{K},E)$, which is isomorphic to the collection of equivalence classes of torsors, often denoted by $\mathrm{WC}(E/K)$, which is called the Weil--Ch\^{a}telet group. Note that $\bigoplus_{v}\mathrm{res}_v$ is well defined. Indeed, if a torsor $C/K$ has good reduction at $v$, then its image in $H^1(G_{K_v}, E)$ vanishes since a genus $1$ curve over a finite field always has a rational point.

The Tate--Shafarevich group is a group that serves as an obstruction to the local-global principle for curves of genus $1$. It is conjectured to be finite (Tate--Shafarevich conjecture). 
By contrast, the $n$-torsion subgroup of $\Sha(E/K)$, that is, 
$\Sha(E/K)[n]\stackrel{\mathrm{def}}{=}\{a\in \Sha(E/K)\mid na=0\}=\mathrm{Ker} \left(H^1(G_K, E)[n] \stackrel{\mathrm{res}}{\to} \bigoplus_{v\in M_K} H^1(G_{K_v}, E)[n]\right)$ is known to be finite. 

 Let $K(\sqrt{D})/K$ be a quadratic extension. In this paper, we investigate the following question.

 \begin{unnumberedquestion} What are the behaviors of $\#\Sha(E/K(\sqrt{D}))[2]$ as a function of $D$? And how do they relate to the behaviors of $\#\Sha(E_D/K)[2]$? Here, $E_D/K$ is the quadratic twist of $E/K$ by $D$. \end{unnumberedquestion}

In other words, this question examines the increase or decrease of counterexamples to the local-global principle under field extensions. When a torsor $[C/K] \in \Sha(E/K)$ acquires an $K(\sqrt{D})$-rational point through the field extension $K(\sqrt{D})/K$, the image of $[C/K]$ in $\Sha(E/K(\sqrt{D}))$ becomes $0$. However, the number of counterexamples to the local-global principle may increase over $K(\sqrt{D})$, making the behavior of $\#\Sha(E/K(\sqrt{D}))$ relative to $\#\Sha(E/K)$ intricate.

First, we review the literature concerning the increasing direction. Rohrlich proved that $\Sha(E_D/\mathbb{Q})[2]$ can become arbitrarily large as $D$ varies \cite{H}. Clark and Sharif proved that $\Sha(E/L)[n]$ can be made arbitrarily large by choosing an appropriate degree-$n$ extension $L/K$, where $n \ge 2$ is a fixed integer (Theorem 3 of \cite{C}). For $n=2$ and $K=\mathbb{Q}$, Matsuno gave an alternative proof of the result of Clark and Sharif \cite{M}. We prove that we can take $D$ such that $\dfrac{\#\Sha(E/\mathbb{Q}(\sqrt{D}))[4]}{\#\Sha(E_D/\mathbb{Q})[2]}\to \infty$ and $\#\Sha(E_D/\mathbb{Q})[2]\to \infty$.

\begin{theorem}[Theorem \ref{cor}]
 For an arbitrary $r\in \Bbb{Z}$ and an elliptic curve over $\Bbb{Q}$ with $E(\Bbb{Q})[2]\cong \Bbb{Z}/2\Bbb{Z}$ that does not have a cyclic 4-isogeny defined over $\Bbb{Q}(E[2])$, there exists a square-free integer $D$ such that
 
 $\dfrac{\#\Sha(E/\Bbb{Q}(\sqrt{D}))[4]}{\#\Sha(E_D/\Bbb{Q})[2]}\ge r$ and $\#\Sha(E_D/\Bbb{Q})[2]\ge r$.
\end{theorem}

The method of the proof is to investigate the 2-torsion subgroup of Yu's formula \cite{Yu}. See Remark \ref{difficult1}. By the proof of Theorem 4.10 in \cite{O2}, we can make $\Sha(E/\Bbb{Q}(\sqrt{D}))[2]$ large while making $\#\mathrm{Sel}^2(E_D/\Bbb{Q})$ smaller than  a constant. We prove that  we can make $\dfrac{\#\Sha(E/\Bbb{Q}(\sqrt{D}))[4]}{\#\Sha(E_D/\Bbb{Q})[2]}$ large while making $\mathrm{Sel}^2(E_D/\Bbb{Q})$ large but keeping $\mathrm{rank}(E_D/\Bbb{Q})$ to be $0$.

The following Theorem \ref{Zhe} plays an important role both in the proof of Theorem \ref{cor} and in our investigation of the 2-torsion subgroup version of Yu's formula.

\begin{theorem}[Theorem \ref{Zhe}]
Let $E/K$ be an elliptic curve over $K$ and $n$ be a positive integer. Then, 
\[\begin{array}{ccc}
X \coloneqq
\mathrm{Coker}\bigg(H^1(G_K,E)[n] & \stackrel{\bigoplus_v{\mathrm{res}_v}}{\longrightarrow} & \bigoplus_{v\in M_K} H^1(G_{K_v},E)[n]\bigg)\\

\hspace{2cm}\rotatebox{90}{$\in$} & &\rotatebox{90}{$\in$} \\

\hspace{2cm} \left[f\right] & \longmapsto & (\left[ f\mid_{G_{K_v}}\right])_v

\end{array}\]
is a finite group and $\#X \le \#
\mathrm{Sel}^n(E/K)$.
When $n$ is a prime number, $\#X=\#\mathrm{Sel}^n(E/K)$ holds. 

\end{theorem}

In Cassels' "Arithmetic of curves of genus 1" parts I-VIII, the $X$ in the theorem mentioned above was represented and studied using the Cyrillic letter $\Zhe$ (see Appendix $2$ of \cite{Ca}). In Appendix $2$ of \cite{Ca}, it is stated that there exists a duality between $\Zhe$ and $\mathrm{Sel}^n(E/K)$ when $n$ is a prime number. With an eye towards potential generalizations for arbitrary $n$, we have provided a detailed proof of this result using arguments that build upon the proof of the Cassels--Poitou--Tate duality (1.5. of Chapter 1 \cite{Coates}) and Theorem \ref{vanish}.

\vskip\baselineskip

Regarding the possibility of decreasing $\Sha(E_D/\Bbb{Q})[2]$, we consider the elliptic curves of the form $E\colon y^2=x^3+px$ where $p$ is a prime number. Klagsbrun's work [Theorem 1.1, \cite{Klagsbrun}] or Smith's recent work on the distribution of $\mathrm{Sel}^2(E_D/\Bbb{Q})$[Theorem 1.5, \cite{Smith}]
imply there exist infinitely many square-free integers $D$ such that $\mathrm{Sel}^2(E_D/\Bbb{Q})\cong \Bbb{Z}/2\Bbb{Z}$.
Let $-D$ be a prime number that satisfies appropriate congruence conditions modulo $4$ and $p$. We prove that $\Sha(E_D/\Bbb{Q})[2]=0$,  under the assumption that the Tate--Shafarevich group is finite. See Proposition \ref{mainlemma}. For elliptic curves of the form $y^2=x^3+px$ with $p\equiv 1\bmod 4$ is a prime number, we also prove that there exists a $D$ with $-D$ being a prime such that $\#\Sha(E/\Bbb{Q}(\sqrt{D}))[2]\le 4 $ and $\Sha(E_D/\Bbb{Q})[2]=0$ under the assumption that the Tate--Shafarevich group is finite. We also prove that $\Sha(E/\Bbb{Q}(\sqrt{-D}))[2]$ cannot be made trivial for any $D$ when $p=257$. See example \ref{impossible}.
\vskip\baselineskip

\begin{proposition}[Proposition \ref{main theorem 3}]
Let $p\equiv 1 \bmod 4$ be a prime number, and let $E:y^2=x^3+px$ be an elliptic curve.  
\begin{enumerate}
\item 
There exist infinitely many imaginary quadratic fields 
$K=\Bbb{Q}(\sqrt{D})$ with $-D$ being a prime number such that $\#\Sha(E/K)[2] \le 4$ and $\Sha(E_D/\Bbb{Q})[2]=0$ under the assumption that $\#\Sha(E/K)$ and $\#\Sha(E_D/\Bbb{Q})$ are finite. 
\item 

If $\Sha(E/\Bbb{Q})$ contains an element of order $4$, then for any quadratic number field $K=\Bbb{Q}(\sqrt{D})$, $\Sha(E/K)[2] \neq 0$.
 
\end{enumerate}

\end{proposition}

The structure of the paper is as follows. In Chapter 3, we determine the cokernel $X$ (Theorem \ref{Zhe}). We also calculate the local cohomology and global cohomology associated with quadratic extensions. In Chapter 4, we prove the main theorem by examining the 2-part of Yu's formula using $X$. The mechanism of growth is in the part that increases (local cohomology)/(global cohomology) (called $g(D)$ in Proposition \ref{g(D)}). By evaluating the image of the corestriction map from below (\ref{cores}), we derive that the ratio $\dfrac{\#\Sha(E/\Bbb{Q}(\sqrt{D}))[4]}{\#\Sha(E_D/\Bbb{Q})[2]}$ increases in conjunction with the growth of $g(D)$. In Chapter 5, for elliptic curves of the form $y^2=x^3+px$ (where $p$ is a prime), we investigate through 2-descent calculations how quadratic extensions and twists by primes can reduce the size of the 2-part of the Tate-Shafarevich group.

\section{Notation}
Let us fix the notation as follows: 

\begin{itemize}[label=\textbullet] 
\item $K$: a number field, $O_K$ : ring of integers of $K$.
\item $M_K$: the set of all places of $K$.
\item $K_v$: the completion of $K$ at place $v\in M_K$.
\item $G_L$: the absolute Galois group of a field $L$, that is, $\mathrm{Gal}(\overline{L}/L)$.

\item For an Abelian group $A$ and an integer $n \geq 2$, we define $A[n]$($n$-torsion subgroup of $A$) to be $\{a \in A \mid na = 0\}$ and $nA$ to be $\{na \mid a \in A \}$.

\item For an Abelian group $A$ and a prime number $p$, we define $A[p^{\infty}]$($p$-primary part of $A$) to be $A[p^{\infty}]\coloneqq\{a\in A | \exists n\ge 0, p^n a = 0\}$.

\item For a locally compact group $A$, $A^*$ is the Pontryagin dual. For a group homomorphism $f : A\to B$ between locally compact group $A$ and $B$, $f^* : B^*\to A^*$ is given by $g\mapsto g\circ f$.
\item For a group homomorphism $f: N\to M$, $f(N)$ is the image of $N$ under $f$. 

\item $\mathrm{rank}(E/K)$: the Mordell-Weil rank of elliptic curve $E/K$, $\Delta_E$: discriminant of $E/K$, $E(K)_{\text{tor}}$: torsion subgroup of $E(K)$. 

\item $E_D/K$: the quadratic twist of $E/K$ by a square-free integer $D$. Namely, if $E/K$ is the elliptic curve defined by $y^2=x^3+ax+b$, then the quadratic twist is an elliptic curve given by the equation $E_D : D y^2=x^3+ax+b$. Quadratic twist $E_D$ is isomorphic to $E$ over $K(\sqrt{D})$ but not isomorphic over $K$. We fix an isomorphism $\tau : E(L)\cong E_D(L), (x,y)\mapsto (x,\frac{y}{\sqrt{D}})$.

\item For an elliptic curve $E/K$, another elliptic curve $E'$ and a nonzero isogeny $\phi : E\to E'$, the $\phi$-Selmer group of $E/K$ is defined as follows:
\[\begin{array}{ccc}
\mathrm{Sel}^{\phi}(E/K) \stackrel{\mathrm{def}}{=} \mathrm{Ker}\bigg(H^1(G_K,E[\phi]) & \to & \prod_{v\in M_K} H^1(G_{K_v},E)[\phi]\bigg).
\end{array}\]
When $E=E'$ and $\phi=[n]$ (multiplication-by-$n$ map), we denote its Selmer group by $\mathrm{Sel}^n(E/K)$.

\end{itemize}

There exists an exact sequence
\begin{equation}
0\to E(K)/nE(K)\to \mathrm{Sel}^n(E/K)\to \Sha(E/K)[n]\to 0 \label{basic}
\end{equation}
and $\Sha(E/K)[n]$ is finite since the $n$-Selmer group is finite
(see Theorem 4.2 in Chapter X of \cite{sil}).

\section{Local cohomology and Global cohomology}

\subsection{Switch local to global}\par

In this section, we determine the cokernel of the restriction map $H^1(G_K,E)[2] \to \bigoplus_{v\in M_K} H^1(G_{K_v},E)[2]$
(Theorem \ref{Zhe}). We denote this cokernel as $X$ in
Theorem \ref{Zhe}. We prove that $X$ is isomorphic to the dual of the $2$-Selmer group, which lives in the global cohomology $H^1(G_K, E[2])$.

\begin{theorem}[\mbox{\rm\textit{cf.}} \cite{Neu}, (8.6.10),  Long Exact Sequence of Poitou--Tate]\label{Poitou}

Let $S$ be a nonempty set of places of a number field $K$ and assume that $S$ contains all infinite places of $K$. Let $K_S$ be the maximal unramified extension  of $K$ outside $S$ and $G_S\coloneqq \mathrm{Gal}(K_S/K)$. Let $M$ be a finite $G_S$ module and $M'=\mathrm{Hom}(M, \mu)$ where $\mu$ is the group of roots of unity in ${K_S}^{\times}$. The following $9$-term exact sequence exists:

\[\begin{tikzcd}[sep=small]
% head
0\arrow[r] & 
%1
H^0(G_S,M)\arrow[r, "\alpha"] &
%2
\displaystyle\prod_{v\in S}H^0(G_{K_v},M)\arrow[r] &
%3
H^2(G_S,M')^* \arrow[d,bend left,shift left=8] \\
 & 
%6
H^1(G_S,M')^* \arrow[d,bend right,shift right=10] &
%5
\displaystyle\sideset{}{'}\prod_{v\in S} H^1(G_{K_v},M) \arrow[l] &
%4
H^1(G_S,M) \arrow[l,"\beta"] \\
&
%7
H^2(G_S,M) \arrow[r,"\gamma"] &
%8
\displaystyle\bigoplus_{v\in S} H^2(G_{K_v},M) \arrow[r] &
%9
H^0(G_S,M')^* \arrow[r] &
% tail
0 .
\end{tikzcd}\]

Here, $\alpha$, $\beta$, and $\gamma$ are localization maps and $\prod'_v$ is a restricted product with respect to unramified cohomology $H_{un}^1(G_{K_v},M)$.

\end{theorem}

\begin{theorem}\label{vanish}
Let $S$ be a set of places of $K$ containing all infinite places. 
Let $K_S$ be the maximal unramified extension  of $K$ outside $S$ and $G_S\coloneqq \mathrm{Gal}(K_S/K)$. Let $M$ be a finite $G_S$-module. 

Let $p$ be a prime.
 If $pM=0$ and $\mathrm{dim}_{\Bbb{F}_p}M \le 2$ holds, then \[\mathrm{Ker}\left(H^1(G_S,M)\stackrel{\beta}{\to}     \displaystyle\sideset{}{'}\prod_{v\in S} H^1(G_{K_v},M)\right)=0.\]

\end{theorem}
 
 \begin{proof}
See [\cite{Ha}, Section 7.2, Corollary of Lemma 2].
 \end{proof}

\begin{theorem}\label{Zhe}
 Let $E/K$ be an elliptic curve and $n$ be a positive integer. Then, 
\[\begin{array}{ccc}
X \coloneqq
\mathrm{Coker}\bigg(H^1(G_K,E)[n] & \stackrel{\bigoplus_v{\mathrm{res}_v}}{\longrightarrow} & \bigoplus_{v\in M_K} H^1(G_{K_v},E)[n]\bigg)\\

\hspace{2cm}\rotatebox{90}{$\in$} & &\rotatebox{90}{$\in$} \\

\hspace{2cm} \left[f\right] & \longmapsto & (\left[ f\mid_{G_{K_v}}\right])_v

\end{array}\]
is a finite group and $\#X \le \#
\mathrm{Sel}^n(E/K)$.
When $n$ is a prime number, $\#X=\#\mathrm{Sel}^n(E/K)$ holds. 

\end{theorem}
\begin{proof}

Let $S$ be a finite subset of $M_K$ containing the infinite places of $K$ and the primes of bad reduction for $E/K$.

The exactness of the following sequence forms part of the long exact sequence in the Poitou--Tate duality (Theorem \ref{Poitou}) when we set $M=E[n]$.

\[H^1(G_S, E[n]) \xrightarrow{\beta} \bigoplus_{v\in S} H^1(G_{K_v}, E[n]) \xrightarrow{\beta^* \circ \psi} H^1(G_S, E[n])^*\]

Here, $\psi \coloneqq \prod_{v\in S} \psi_v$, where $\psi_v$ is the isomorphism given by the local Tate duality:\[\psi_v \colon H^1(G_{K_v},E[n]) \cong H^1(G_{K_v},E[n])^*\](see [\cite{Neu}, Theorem 7.2.6]).

Let $\iota$ be the map that makes the following diagram commutative:

\[\begin{tikzcd}[column sep=small, row sep=large]
0 \arrow[r] & H^1(G_S,E[n])\arrow[r,"\beta"] \arrow[d, equals] & \bigoplus_{v\in S} H^1(G_{K_v},E[n])\arrow[r,"h"] \arrow[d, equals] &
\mathrm{Coker}\beta
\arrow[r] 
 \arrow[d,"\iota"] & 0 \\  
0 \arrow[r] & H^1(G_S,E[n]) \arrow[r,"\beta"] & \bigoplus_{v\in S} H^1(G_{K_v},E[n])  \arrow[r,"\beta^*\circ \psi"] &H^1(G_S,E[n])^* .\\
\end{tikzcd}\]
From the above diagram, $\iota$ is injective. Let us consider the following diagram. 

\[
\begin{CD}
@.\mathrm{Ker}\beta @>>>
\mathrm{Ker}\phi_S\\
@. @VVV @VVV\\
0@>>>H^1(G_S,E[n])@=H^1(G_S,E[n])\\
@VVV @VV\beta V @VV\phi_S V \\
\bigoplus_{v\in S}H^1(G_{K_v},E)[n]^*@>      
{\psi}^{-1}\circ \lambda^*>>\bigoplus_{v\in S} {H^1(G_{K_v},E[n])}@>\lambda>>\bigoplus_{v\in S}H^1(G_{K_v},E)[n] \\
@V = VV @V h VV \\
\bigoplus_{v\in S} H^1(G_{K_v},E)[n]^*@>j>>\mathrm{Coker}\beta. \\ \end{CD} \]
Here, $\phi_S \stackrel{\mathrm{def}}{=}\lambda \circ \beta$ and $\lambda\coloneqq \bigoplus_{v\in S} \lambda_v$, where $\lambda_v: H^1(G_{K_v},E[n])\to H^1(G_{K_v}, E)[n]$ is the map induced by a short exact sequence $0\to E[n]\to E\stackrel{\times n}{\to} E\to 0$ (see [\cite{sil}, Section V\hspace{-1.2pt}I\hspace{-1.2pt}I\hspace{-1.2pt}I.2] ).

Note that \[\bigoplus_{v\in S}H^1(G_{K_v},E)[n]^*\stackrel{{\psi}^{-1}\circ \lambda^*}{\to}    
\bigoplus_{v\in S} {H^1(G_{K_v},E[n])}\stackrel{\lambda}{\to}\bigoplus_{v\in S}H^1(G_{K_v},E)[n]\] is an exact sequence 
because \[\mathrm{Im}({\psi_v}^{-1}\circ {\lambda_v}^*)\cong \mathrm{Im}({\lambda_v}^*)\cong (\mathrm{Im}\lambda_v)^*\cong (H^1(G_{K_v},E)[n])^*\cong E(K_v)/nE(K_v)\cong \mathrm{Ker}\lambda_v\] where the isomorphism $H^1(G_{K_v},E)[n])^* \cong E(K_v)/nE(K_v)$ is given by the restricted Tate local duality (see [\cite{tate}, Proposition 1]).

By the snake lemma, we obtain the following exact sequence
\[\mathrm{Ker}\beta \to \mathrm{Ker}\phi_S\to \bigoplus_{v\in S} H^1(G_{K_v},E)[n]^*\stackrel{j}{\to} \mathrm{Coker}\beta.\]
We claim that the following diagram $\textcircled{\scriptsize a}$ is commutative. 

\[
\begin{tikzcd}
\mathrm{Ker}\beta \ar[r]&\mathrm{Ker}\phi_S\ar[r]&\bigoplus_{v\in S} H^1(G_{K_v},E)[n]^*\ar[r, "j"]\ar[dr, "{\phi_S}^*", swap ]&\mathrm{Coker}\beta\ar[d,"\iota", hook']\\
&&&H^1(G_K,E[n])^*
\end{tikzcd}
\cdots \textcircled{\scriptsize a}.\]
Indeed, this follows from the following diagram:
\[   
\begin{tikzcd}[
    every matrix/.append style={name=mycd},
    execute at end picture={
      \path (mycd-2-1) -- node[scale=2] {$\circlearrowleft$} node[scale=.5]{$(2)$} (mycd-3-2);
      \path ([xshift=11em]mycd-1-1) -- node[scale=2] {$\circlearrowright$} node[scale=.5]{$(1)$} ([xshift=15em]mycd-2-1);
      \path (mycd-2-1) -- node[yscale=3.2, xscale=-1.5, rotate=90, yshift=-.1em] {\textasciitilde} (mycd-3-1);
      }
    ]  
    \bigoplus_{v\in S} H^1(G_{K_v},E)[n]^*\ar[r, "j"]\ar[dr, "{\phi_S}^*", swap ]\ar[d,"\lambda^*", swap]& \mathrm{Coker}\beta \ar[d,"\iota", hook']\\ 
    \bigoplus_{v\in S} H^1(G_{K_v},E[n])^* \ar[r,"\beta^*"]\ar[dash, d,swap,
    "\psi", shift left=.5em, outer sep=.3em]&H^1(G_K,E[n])^*\\ \bigoplus_{v\in S} H^1(G_{K_v},E[n])\ar[r, "h"]& \mathrm{Coker}\beta\ar[u, "\iota", hook, swap].   
\end{tikzcd} 
\]
To prove the desired commutativity, it is sufficient to prove the commutativity of $(1)$ and $(2)$ of the above diagram because $h \circ {\psi}^{-1} \circ \lambda^* = j$. The commutativity of $(1)$ follows directly from the definition of $\phi_S$, and $(2)$ is exactly the definition of $\iota$, that is, $\iota \circ h=\beta^* \circ \psi$.

From the commutativity of the diagram $\text{\textcircled{\scriptsize a}}$, we obtain the following exact sequence: 
\[\mathrm{Ker}\beta \to \mathrm{Ker}\phi_S \to \bigoplus_{v\in S} H^1(G_{K_v},E)[n]^*
\stackrel{{\phi_S}^*}{\to}H^1(G_K, E[n])^*.\]

There is a canonical isomorphism $\mathrm{Ker}\phi_S \cong \mathrm{Sel}^n(E/K)$ (see [\cite{mil}, Chapter I, Corollary 6.6]).

By taking its Pontryagin dual, we obtain

\[ H^1(G_K,E[n])\stackrel{\phi_S}{\to} \bigoplus_{v\in S} H^1(G_{K_v},E)[n]\to {\mathrm{Sel}^n(E/K)}^*\to (\mathrm{Ker}\beta)^*.  \]

By taking the direct limit $\varinjlim_{S \subset M_K}$, we obtain the following exact sequence: 
\[0\to \text{Sel}^n(E/K)\to H^1(G_K, E[n])\xrightarrow{\phi}  \bigoplus_{v\in M_K}{H^1(G_{K_v},E)[n]}\stackrel{\epsilon}{\to} {\mathrm{Sel}^n(E/K)}^*.\]

We can conclude that 
$\#X\le \#\mathrm{Sel}^n(E/K)$ from the following commutative diagram. 

{\footnotesize\begin{tikzcd}[column sep=small, row sep=large]
    0 \arrow[r] & \text{Sel}^n(E/K) \arrow[r] \arrow[d, ->>] & H^1(G_K,E[n]) \arrow[r,"\phi"] \arrow[d, ->>] & 
    \bigoplus_{v\in M_K}{H^1(G_{K_v},E)[n]}    
    \arrow[r, "\epsilon"] \arrow[d, equals] & \text{Sel}^n(E/K)^* \\
    0 \arrow[r] & \Sha(E/K)[n] \arrow[r] & H^1(G_K,E)[n] \arrow[r] & \bigoplus_{v\in M_K} H^1(G_{K_v},E)[n] \arrow[r] & X \arrow[r] & 0 . 
\end{tikzcd}}

Indeed, since $X\cong \mathrm{Coker}\phi$ injects into ${\mathrm{Sel}^n(E/K)}^*$, it follows that $\#X\le \#\mathrm{Sel}^n(E/K)$. 

When $n$ is a prime number, $\mathrm{Ker}\beta=0$ by Theorem \ref{vanish}. Therefore, $\epsilon$ is surjective, hence $\#X=\#\mathrm{Sel}^n(E/K)$.\end{proof}

\subsection{Local vs Global}
 
The local cohomology of elliptic curves at primes of bad reduction (especially those with additive reduction) remains mysterious, but we can calculate it at primes of good reduction. In this section, we prove that we can increase the contributions from local cohomologies while suppressing the contributions from global cohomology.

\renewcommand{\labelenumi}{(\arabic{enumi})}

\begin{proposition}\label{local}
Let $L$ be a quadratic extension of $K$.
Let $v\in M_K$ be a place that is not above $2$ and let $w$ a place of $L$ above $v$. Let $E/K$ be an elliptic curve.
Let $k$ be the residue field of $K_v$ and $\tilde{E}/k$ be the reduction of $E/K$ mod $v$. \par
\begin{enumerate}

\item Suppose $v$ is a finite place of $K$ unramified in $L$, and $E/K$ has good reduction at $v$. Then $H^1(\mathrm{Gal}(L_w/K_v), E(L_w)) = 0$.

\item $H^1(\mathrm{Gal}(L_w/K_v), E(L_w))$ is a finite group. 
\item Suppose $v$ is a finite place of $K$ ramified in $L$, and $E/K$ has good reduction at $v$. Then \[\#H^1(\mathrm{Gal}(L_w/K_v), E(L_w)) = \#\tilde{E}(k)[2].\]

\end{enumerate}
\end{proposition}
\begin{proof}

(1) See [\cite{Mazur}, Corollary 4.4].\par
(2) We have an inflation-restriction exact sequence: \[0\to H^1(\text{Gal}(L_w/K_v), E(L_w))
\xrightarrow{\mathrm{inf}}
H^1(G_{K_v},E)\xrightarrow{\mathrm{res}}H^1(G_{L_w},E).\]
From Tate duality, the dual of this sequence is \[ E(L_w)\stackrel{\mathrm{tr}}{\to}E(K_v)\to H^1(\mathrm{Gal}(L_w/K_v),E(L_w))^*\to 0.\]
Thus, $H^1(\mathrm{Gal}(L_w/K_v),E(L_w))\cong E(K_v)/\mathrm{tr}(E(L_w))$ holds (see [\cite{Mazur}, Proposition 4.2] for this isomorphism and [\cite{tate}, equation (12)] for the relation ${\mathrm{res}}^*=\mathrm{tr}$).

There exists a surjective map from the weak Mordell--Weil group $\frac{E(K_v)}{2E(K_v)}$ to $\frac{E(K_v)}{\mathrm{tr}(E(L_w))}$. Because $\frac{E(K_v)}{2E(K_v)}\subset H^1(G_{K_v}, E[2])$, it is sufficient to prove $H^1(G_{K_v}, E[2])$ is finite. Let $M=K_v(E[2])$. Then $M/K_v$ is a finite Galois extension. Because $H^1(G_M,E[2])\cong (M^{\times}/{M^{\times}}^2)^2$ is finite and $\#H^1(\text{Gal}(M/K_v), E(M)[2])$ is finite, we see that $H^1(G_{K_v}, E[2])$ is finite because of the exact sequence $0\to H^1(\text{Gal}(M/K_v), E(M)[2])\xrightarrow{\mathrm{inf}}H^1(G_{K_v},E[2])\xrightarrow{\mathrm{res}}H^1(G_{M},E[2])$.\par

(3) From $(2)$, it is sufficient to prove $\#\dfrac{E(K_v)}{\mathrm{tr}(E(L_w))}=\#\tilde{E}(k)[2]$. 
There exists an exact sequence: \[0\to \dfrac{{E}_1(K_v)}{\mathrm{tr}({E}_1(L_w))}\to \dfrac{E(K_v)}{\mathrm{tr}(E(L_w))}\stackrel{\mathrm{reduction}}{\to} \dfrac{\tilde{E}(k)}{2\tilde{E}(k)}\to 0 \] where $E_1$ denotes the kernel of reduction. Note that $\mathrm{reduction}$ is well-defined since $L_w/K_v$ is a ramified extension. 
Let us prove that the group on the left-hand side is trivial.
Indeed, \( E_1(K_v) \cong \hat{E}(M) \), where \( M \) is the maximal ideal of \( O_K \) and \( \hat{E}(M) \) is the group associated with the formal group \( \hat{E} \), which is a $2$-divisible group. For all \( a \in E_1(K_v) \), there exists \( b \in E_1(K_v) \) such that \( a = 2b = \mathrm{tr}(b) \). Then we can conclude that $\#H^1(\mathrm{Gal}(L_w/K_v),E(L_w)))=\#\dfrac{E(K_v)}{\mathrm{tr}(E(L_w))}=\#\dfrac{\tilde{E}(k)}{2\tilde{E}(k)}=\#
\tilde{E}(k)[2]$.\end{proof}

\begin{remark}
When $L/K$ is a quadratic extension and there is a choice of $w \in M_L$ above $v \in M_K$, $v$ splits completely. Since $H^1(\mathrm{Gal}(L_w/K_v), E(L_w)) = 0$ in this case, the choice of $w$ is not an issue. More generally, for a Galois extension $L/K$ of degree $n$, it can be shown that $H^1(\mathrm{Gal}(L_w/K_v), E(L_w))$ is independent of the choice of $w \mid v$. Similarly to (1) and (2) of Proposition \ref{local}, it can be proven that $H^1(\mathrm{Gal}(L_w/K_v), E(L_w))$ is a finite group that vanishes for almost all $v$.
\end{remark}

The following theorem of Qiu \cite{Qiu}, using computations with the Herbrand quotient, implies that the global cohomology $H^1(\mathrm{Gal}(L/K),E(L))$ remains small, provided that $\mathrm{rank}(E/L)=\mathrm{rank}(E/K)+\mathrm{rank}(E_D/K)$ does not grow significantly.

\begin{theorem}\label{Qiu}\par
Let $E/K$ be an elliptic curve over $K$. Let $L=K(\sqrt{D})$ be a quadratic extension of $K$. Let $\sigma$ be a generator of $\mathrm{Gal}(L/K)$.
We define $\mathrm{tr}: E(L)\to E(K)$ by $P\mapsto P+P^{\sigma}$. Then, \[\#H^1(\mathrm{Gal}(L/K),E(L))=\#\mathrm{Coker}(\mathrm{tr})\times 2^{\mathrm{rank}(E_D/K)-\mathrm{rank}(E/K)}\] holds. 
In particular, $\#H^1(\mathrm{Gal}(L/K),E(L))\le 2^{\mathrm{rank}(E_D/K)}\times \#
E(K)[2]$
\end{theorem}
\begin{proof}
For the former, see [\cite{Qiu}, Theorem 1.5].
The latter immediately follows from the fact that there is a surjection from $E(K)/2E(K)$ to $\mathrm{Coker}(\mathrm{tr})$. \end{proof}

The following theorem by Rohrlich ensures that we can make $\mathrm{rank}(E_D/K)$ small when $K=\Bbb{Q}$.

\begin{theorem}[\mbox{\rm\textit{cf.}}\ J.Hoffstein and W.Luo  appendixed by Rohrlich \cite{H}]\label{H}
 For any elliptic curve $A$ over $\Bbb{Q}$, there exist infinitely many square-free integers $D\in \Bbb{Z}$ such that $\mathrm{rank}(A_D/\Bbb{Q})=0$ and the number of prime factors of $D$ is no greater than $4$. 
\end{theorem}

\begin{proposition}\label{g(D)}
 Let $K$ be a number field. Suppose that there exists a quadratic extension $K(\sqrt{d})$ of $K$ such that $\mathrm{rank}(A_d/K)=0$ for arbitrary elliptic curve $A$ over $K$ and the number of prime factors of $d$ is no greater than some constant $a$. 
   Then for an arbitrary integer $r \in \Bbb{Z}$ and arbitrary elliptic curve $E/K$, there exist infinitely many quadratic extensions $L=K(\sqrt{D})/K$ such that 
 
 \[g(D)\coloneqq \dfrac{\#\bigoplus_{v\in M_K}H^1(\mathrm{Gal}(L_w/K_v),E(L_w))}{\#H^1(\mathrm{Gal}(L/K), E(L))}\ge r.\]

\end{proposition}

\begin{proof}

 By Theorem \ref{Qiu}, it is sufficient to prove there exists infinitely many $D$ such that 
 \[\dfrac{\#\bigoplus_{v\in M_K}H^1(\mathrm{Gal}(L_w/K_v),E(L_w))}{2^{\text{rank}(E_D/K)}}\ge r.\]

 For an arbitrary $r$, there exists an integer $R$ such that $4^{R-a} \geq r$. For this $R$, there exist prime elements $v_i (1\le i\le R)$ of $O_K$ such that 
 \begin{enumerate}
 \item principal ideals $(v_i)$ split completely in $K(E[2])$.
 \item $E/K$ has good reduction at $v_i$, and $v_i$ does not lie above $2$.
 \end{enumerate}
 
Indeed, the prime ideals of $O_K$ that split completely in the Hilbert class field  $K^H$ of $K$ are precisely the principal prime ideals of $O_K$, and we can take infinitely many prime elements $v \in O_K$ that satisfy conditions (1) and (2) above by applying Chebotarev's density theorem to $K^HK(E[2])$. For such a prime element $v$ of $O_K$, $K_{v}(E[2]) =K_{v}$ since $v$ splits completely in $K(E[2])$, thus $E(\overline{K_v})[2]=E(K_v)[2]$ injects into $\tilde{E}(k_{v})[2]$, and therefore $\#
\tilde{E}(k_{v})[2]=4$ where $k_v$ is the residue field of $K_v$. Thus if $v$ is a ramified prime element of a quadratic extension $L=K(\sqrt{D})$ of $K$, $\#H^1(\mathrm{Gal}(L_w/K_v),E(L_w))=4$ holds by Proposition \ref{local} (3). By setting $A = E_{v_1 v_2 \cdots v_R}$, there exists a quadratic extension $K(\sqrt{D_R})/K$ such that 
$\text{rank}(E_{v_1 v_2 \cdots v_R D_R}/K) = 0$ and the number of prime factors of $D_R$ is no greater than $a$. 
 Let us take  $D$ as $D=v_1\cdots v_R D_R$. Because the number of prime factors of $D_R$ is no greater than $a$, the number of ramified prime elements of $K$ under quadratic extension $L=K(\sqrt{v_1 v_2 \cdots v_R D_R})/K$ which satisfies condition $(1),(2)$ is $R-a$ or more. Thus, taking $L = K(\sqrt{v_1 v_2 \cdots v_R D_R})$,

\[
\dfrac{\#\bigoplus_{v\in M_K} H^1(\mathrm{Gal}(L_w/K_v), E(L_w))}{2^{\operatorname{rank}(E_D/K)}}
\ge \dfrac{4^{R-a}}{2^{\operatorname{rank}(E_D/K)}} \ge \dfrac{4^{R-a}}{2^0} \ge r.\ \qedhere
\]
\end{proof}

\begin{remark} \label{localglobal}
In particular, when $K=\Bbb{Q}$, Theorem \ref{H} implies that for all $r \in \Bbb{Z}$ and all elliptic curves $E/K$, there exist infinitely many square-free integers $D$ such that $g(D)\ge r$. If we can determine that $\mathrm{rank}(E_D/K)$ does not grow significantly compared to $4^{\omega(\Delta_{E_D})}$, where $\omega(\Delta_{E_D})$ is the number of prime factors of $\Delta_{E_D}$, this result can be generalized to cases where $K \neq \mathbb{Q}$.
\end{remark}

\section{Increasing $\Sha(E/\Bbb{Q}(\sqrt{D}))[2]$ and $\Sha(E_D/\Bbb{Q})[2]$ simultaneously}
\subsection{Trace and twist}
 In this section, we investigate the relationship between $\mathrm{tr}(\Sha(E/L)[2])$ and $\#\Sha(E_D/K)[2]$. The Galois group naturally acts on $\Sha(E/L)$ by acting on the coefficients of torsors, and we can explicitly calculate this action cohomologically.

\begin{definition}[$\mathrm{Gal}(L/K)$ acts on $\Sha(E/L)$]\label{def1}
Let $G$ be a group and $H$ be a normal finite index subgroup of $G$. Let $M$ be a $G$-module. 
$G/H$ acts on $H^1(H, M)$ by $(\overline{g} * X)(h) = gX(g^{-1}hg)$. Here, $g$ is a lift of $\bar{g} \in G/H$ in $G$, and $gX(g^{-1}hg)$ does not depend on the lift modulo coboundary.

When \( G = G_K \) and \( H = G_L \) and $M=E$, the Galois group \( \mathrm{Gal}(L/K) \cong G/H \) acts on \( \Sha(E/L)[2] \subset H^1(G_L, E) \) as described above.
\end{definition}

\begin{definition}[Corestriction and trace]  

In the case $[G:H]=2$, 
let $\sigma$ be a generator of $G/H$. There is a map $\mathrm{cor} : H^1(H,M)\to H^1(G,M)$ called corestriction  which satisfies  $\mathrm{cor}\circ \mathrm{res}=[2]$ and  $\mathrm{tr}=\mathrm{res}\circ \mathrm{cor}$ where  $\mathrm{tr} : H^1(H,M)\to H^1(H,M)$ be a map defined by $X\mapsto  X+\sigma*X$.

 Explicitly, $\mathrm{cor} : H^1(H,M)\to H^1(G,M)$ can be expressed as follows.

{\footnotesize
\[
H^1(H, M) \to H^1(G, M) \text{ is given by } X \mapsto \left[ 
g \mapsto 
\begin{cases}
X(g) + (\sigma * X)(g) & \text{if } g \in H, \\
X(g \sigma) + (\sigma * X)(g \sigma) & \text{if } g \in G - H
\end{cases}
\right].\]}

This is indeed a trace map if we restrict it to $H^1(H,M)$.
When $G=G_K, H=G_L, M=E$, the map $\mathrm{cor}: H^1(G_L,E)\to H^1(G_K,E)$ and $\mathrm{tr}: H^1(G_L,E)\to H^1(G_L,E)$ induce maps $\mathrm{cor}: \Sha(E/L)[2]\to \Sha(E/K)[2]$ and $\mathrm{tr}: \Sha(E/L)[2]\to \Sha(E/L)[2]$.

 \end{definition}

\begin{proposition}\label{tr and cor}
 $\#\mathrm{tr}(\Sha(E/L)[2])\ge \dfrac{\#\mathrm{cor}(\Sha(E/L)[2])}{\#H^1(\mathrm{Gal}(L/K),E(L))}$.

\end{proposition}

\begin{proof}
Since \[\mathrm{tr}=\mathrm{res}\circ \mathrm{cor},\] there exists a map $\mathrm{cor}(\Sha(E/L)[2])\stackrel{\mathrm{res}}{\to} \mathrm{tr}(\Sha(E/L)[2])$.
By the inflation-restriction sequence, $\mathrm{Ker}(\mathrm{res})$ is contained in $H^1(\mathrm{Gal}(L/K), E(L))$. By applying the first isomorphism theorem to $\mathrm{res}$, we obtain the inequality.
\end{proof}

\begin{theorem}\label{Yu}
 Let
\begin{align*}
\langle\ ,\rangle_K &: \Sha(E/K)[2]\times \Sha(E/K)[2]\to \mathbb{Z}/2\mathbb{Z},
\\
\langle\ ,\rangle_L &: \Sha(E/L)[2]\times \Sha(E/L)[2]\to \mathbb{Z}/2\mathbb{Z}
\end{align*}
be the Cassels--Tate pairing. Then,
\begin{itemize}
\item[(1)] The kernel on each side is $\Sha(E/K)[2]\cap 2\Sha(E/K)=2\Sha(E/K)[4]$ and $2\Sha(E/L)[4]$, respectively.

\item[(2)]$\langle a, \mathrm{cor}(a')\rangle_K=\langle  \mathrm{res}(a),a'\rangle_L$.

\end{itemize}
 
\end{theorem}
\begin{proof}
For (1), see [\cite{tate}, Theorem 3.2]. See also [\cite{Wuthrich}, Theorem 15]. For (2), see [\cite{Yu},Theorem  8]. 
\end{proof}

\begin{cor}\label{cores}
  \[\#\mathrm{cor}(\Sha(E/L)[2])\ge \dfrac{\#\Sha(E/K)[2]}{\#2 \Sha(E/L)[4]\# H^1(\mathrm{Gal}(L/K), E(L))}.\]

\end{cor}
\begin{proof}
By Theorem \ref{Yu} (1), the Cassels--Tate pairing induces  non-degenerate pairings 

\begin{align*}
\langle \ ,\ \rangle'_K &: \frac{\Sha(E/K)[2]}{2\Sha(E/K)[4]} \times \frac{\Sha(E/K)[2]}{2\Sha(E/K)[4]} \to \mathbb{Z}/2\mathbb{Z}, \\
\langle \ ,\ \rangle'_L &: \frac{\Sha(E/L)[2]}{2\Sha(E/L)[4]} \times \frac{\Sha(E/L)[2]}{2\Sha(E/L)[4]} \to \mathbb{Z}/2\mathbb{Z}.
\end{align*}

Let $\mathrm{res'}:\dfrac{\Sha(E/K)[2]}{2\Sha(E/K)[4]}\to \dfrac{\Sha(E/L)[2]}{2\Sha(E/L)[4]}$ and $\mathrm{cor'}:\dfrac{\Sha(E/L)[2]}{2\Sha(E/L)[4]}\to \dfrac{\Sha(E/K)[2]}{2\Sha(E/K)[4]}$ be the map induced by $\mathrm{res}: \Sha(E/K)[2]\to \Sha(E/L)[2], \mathrm{cor}: \Sha(E/L)[2]\to \Sha(E/K)[2]$ respectively.

By Theorem \ref{Yu} (2), $\langle a+2\Sha(E/K)[4], \mathrm{cor'}(a'+2\Sha(E/L)[4])\rangle'_K\coloneq \langle a, \mathrm{cor}(a')\rangle_K=\langle  \mathrm{res}(a),a'\rangle_L\eqqcolon\langle  \mathrm{res'}(a+2\Sha(E/K)[4]),a'+2\Sha(E/L)[4]\rangle'_L$ holds.

Therefore, the following diagram is commutative, where the vertical map is an isomorphism induced by 
$\langle \ ,\ \rangle'_L$. Note that these two vertical arrows are isomorphisms because the induced pairing is non-degenerate.

\[
  \begin{CD}
    \dfrac{\Sha(E/K)[2]}{2\Sha(E/K)[4]}  @>\mathrm{res'}>>\dfrac{\Sha(E/L)[2]}{2\Sha(E/L)[4]}\\
    @V \cong VV @VV\cong V \\
    (\dfrac{\Sha(E/K)[2]}{2\Sha(E/K)[4]})^* @>>\mathrm{cor'}^*>(\dfrac{\Sha(E/L)[2]}{2\Sha(E/L)[4]})^*
  \end{CD}
\]

From this commutative diagram, 
\begin{equation}\#\mathrm{cor'}(\dfrac{\Sha(E/L)[2]}{2\Sha(E/L)[4]})=\#\mathrm{cor'}^*(({\dfrac{\Sha(E/K)[2]}{2\Sha(E/K)[4]}})^*)=\#\mathrm{res'}(\dfrac{\Sha(E/K)[2]}{2\Sha(E/K)[4]})\label{dual} \end{equation} holds.
Let us consider the following diagram. For simplicity, we abbreviate $\mathrm{res}: \Sha(E/K)[2]\to \Sha(E/L)[2]$ as $f$, and denote by $g$ the map induced by $f$ on $2\Sha(E/K)[4]$.
\vskip\baselineskip

\hspace*{-15mm}
\begin{tikzcd}[
  column sep={7em,between origins}, % 6emから7emに調整
  row sep=2ex,
  bezier bounding box
]
 &
 &0\arrow[d]
 &0\arrow[d]
 &0\arrow[d]
 \\
 &
0 \arrow[r, shorten >=0.2em] & \mathrm{Ker}g \arrow[r, shorten >=0.2em] \arrow[d] &
 2\Sha(E/K)[4] \arrow[r,"g"] \arrow[d] &
2\Sha(E/L)[4] \arrow[r] \arrow[d] &
\mathrm{Coker}g\\  
 &
0 \arrow[r, shorten >=0.2em] & \mathrm{Ker}f \arrow[r, shorten >=0.5em] \arrow[d] &
\Sha(E/K)[2] \arrow[r,"f"] \arrow[d] &
\Sha(E/L)[2] \arrow[d]&
\\ &
0 \arrow[r, shorten >=0.5em] & \mathrm{Ker} (\mathrm{res'}) \arrow[r, shorten >=0.5em] \arrow[rrruu, in=140, out=140, looseness=1.6,"\delta"] &
\dfrac{\Sha(E/K)[2]}{2\Sha(E/K)[4]} \arrow[d] \arrow[r,"\mathrm{res'}"] &
\dfrac{\Sha(E/L)[2]}{2\Sha(E/L)[4]} \arrow[d] & \\
&&  & 0 & 0& \\
\end{tikzcd}

By applying the snake lemma to this diagram, we obtain an exact sequence: \[0\to \mathrm{Ker}g\to \mathrm{Ker}f\to \mathrm{Ker}(\mathrm{res'})\to \mathrm{Im}\delta\to 0.\]

Because  $\mathrm{Ker}f \subset H^1(\mathrm{Gal}(L/K), E(L))$, 
we obtain

{\small \begin{equation}
\begin{aligned}
    \#\mathrm{Ker}(\mathrm{res'}) &= \frac{\#\mathrm{Ker}f \, \#\mathrm{Im}\delta}{\#\mathrm{Ker}g} \\
    &\le \frac{\#H^1(\mathrm{Gal}(L/K), E(L)) \, \#\mathrm{Coker}g}{\#\mathrm{Ker}g}=\frac{\#H^1(\mathrm{Gal}(L/K), E(L)) \, \#2\Sha(E/L)[4]}{\#2\Sha(E/K)[4]}.
\end{aligned}
\label{snake}
\end{equation}}

Here, the last equality holds because of the exactness of \[0\to \mathrm{Ker}g\to 2\Sha(E/K)[4]\stackrel{g}{\to} 2\Sha(E/L)[4]\to \mathrm{Coker}g\to 0.\]
Therefore, 
\begin{align*}
\#\mathrm{cor}(\Sha(E/L)[2])\ge \#\mathrm{cor'}(\dfrac{\Sha(E/L)[2]}{2\Sha(E/L)[4]})&=\#\mathrm{res'}(\dfrac{\Sha(E/K)[2]}{2\Sha(E/K)[4]})\ (\text{by}\ \eqref{dual})\\
&= \frac{\#\Sha(E/K)[2]}{\#\mathrm{Ker}(\mathrm{res'})\#2\Sha(E/K)[4]} \\
&\ge \frac{\#\Sha(E/K)[2]}{ \#2\Sha(E/L)[4]\#H^1(\mathrm{Gal}(L/K),E(L))}\ .
\end{align*}

The last inequality follows from \eqref{snake}.\end{proof}

\begin{proposition}\label{trace lemma}
Let $L/K$ be a quadratic extension and $L=K(\sqrt{D})$. Then,  
 \[\#\mathrm{tr}(\Sha(E/L)[2])\ge \dfrac{\#\Sha(E_D/K)[2]}{4^{\mathrm{rank}(E/K)}\#E(K)[2]^2\#
 2 \Sha(E/L)[4]}\] holds. 

\end{proposition}
\begin{proof}
   Let $\phi : \Sha(E/L)\cong \Sha(E_D/L)$ be an isomorphism induced by $\tau : E(L)\cong E_D(L), (x,y)\mapsto (x,\frac{y}{\sqrt{D}})$. The following commutative diagram exists. 
\[
  \begin{CD}
    \Sha(E/L)  @>\mathrm{tr}>> \Sha(E/L)\\
    @V \phi VV @VV\phi V \\
    \Sha(E_D/L) @>>1-\sigma> \Sha(E_D/L)
  \end{CD}
\]
From this diagram, $\#
\mathrm{tr}(\Sha(E/L)[2])= \#\mathrm{tr}(\Sha(E_D/L)[2])$ for $2$-torsions.

By applying Proposition \ref{tr and cor}, Corollary \ref{cores}, and Theorem \ref{Qiu}
, we obtain the inequality.\end{proof}

\subsection{Main theorem}

Let us consider the inflation-restriction sequence 
\[0 \to H^1(\mathrm{Gal}(L/K), E(L)) \stackrel{\mathrm{inf}}{\to} H^1(G_K, E) \stackrel{\mathrm{res}}{\to} H^1(G_L, E)^{\mathrm{Gal}(L/K)}.\]
Taking the $2$-torsion subgroup of this sequence , we obtain an exact sequence
\[0 \to H^1(\mathrm{Gal}(L/K), E(L)) \stackrel{\mathrm{inf}}{\to} H^1(G_K, E)[2] \stackrel{\mathrm{res}}{\to} H^1(G_L, E)[2]^{\mathrm{Gal(L/K)}}.\] Note that $H^1(\mathrm{Gal}(L/K),E(L))$ is a $2$-torsion group, that is, $2H^1(\mathrm{Gal}(L/K),E(L))=0$ because $\#
\mathrm{Gal}(L/K)=2$.

Consider the local-global version of this diagram, and draw the following diagram with exact rows and columns:

\tikzcdset{
  arrow style=tikz,
  diagrams={>=Straight Barb}
}

\begin{figure}[H]
\centering

%% local command for shorter input
\newcommand{\bop}[1]{%
  \mathop{\smash[b]{\bigoplus\limits_{#1}}}%
}

\hspace*{-18mm}
 \begin{tikzcd}[column sep=2em, /tikz/font=\normalsize]
&[-0.5em] 0 \arrow[d] & 0 \arrow[d] & 0 \arrow[d] &[-1.5em]
\\
& \mathrm{Ker} F\arrow[r] \arrow[d]
& \Sha(E/K)[2] \arrow[r] \arrow[d]
& \mathrm{Ker} H \arrow[d]
\\
& H^1(\mathrm{Gal}(L/K),E(L)) \arrow[r,"\text{inf}"] \arrow[d, "F"]
& H^1(G_K,E)[2] \arrow[r,"\text{res}"] \arrow[d,"G"]
& \mathrm{res}(H^1(G_K,E)[2]) \arrow[d,"H"] \arrow[r]
& 0
\\
0 \arrow[r]
& \bop{v\in M_K} H^1(\mathrm{Gal}(L_w/K_v),E(L_w))
    \arrow[r] \arrow[d]
& \bop{v\in M_K} H^1(G_{K_v},E)[2]
    \arrow[r] \arrow[d]
& \bigoplus\limits_{v\in M_K} \mathrm{res}(H^1(G_{K_v},E)[2])
\\
& \mathrm{Coker} F \arrow[r,"j"] \arrow[d]
& X \arrow[d]
\\
& 0 & 0 .
\end{tikzcd}

\caption{A key diagram}\label{figure1}

\end{figure}

Notably, by Theorem \ref{Zhe}, $X\cong \mathrm{Sel}^2(E/K)^*$ holds.

\begin{lemma}\label{1}
$\#\mathrm{Ker}H= \dfrac{\#\Sha(E/K)[2]\#\bigoplus_{v\in M_K} H^1(\mathrm{Gal}(L_w/K_v),E(L_w))}{\#j(\mathrm{Coker}F)\#H^1(\mathrm{Gal}(L/K),E(L))}$. 

\end{lemma}
\begin{proof}
By applying the snake lemma, 
 \[0\to \mathrm{Ker}F\to \Sha(E/K)[2]\to \mathrm{Ker}H \to \mathrm{Coker}F\to j(\mathrm{Coker}F) \to 0.\]

We obtain $\#\mathrm{Ker}H= \dfrac{\displaystyle \#\Sha(E/K)[2]\#\mathrm{Coker}F}{\displaystyle \#j(\mathrm{Coker}F)\#\mathrm{Ker}F}$. 
The left vertical exact sequence implies $\dfrac{\#\mathrm{Coker}F}{\#\mathrm{ker}F}=\dfrac{\#\bigoplus_{v\in M_K}H^1(\mathrm{Gal}(L_w/K_v),E(L_w))}{\#H^1(\mathrm{Gal}(L/K),E(L))}$, thus the proposition follows. \end{proof}

In Figure \ref{figure1}, $\mathrm{Ker}H$ lives in $\Sha(E/L)[2]$. Henceforth, we shall prove that the ratio of the order of $\Sha(E/L)[2]$ to that of $\mathrm{Ker}H$ is related to $\#\Sha(E_D/K)[2]$.

\begin{figure}[H]
\centering

% local command to ease input
\newcommand{\edgedesc}[1]{%
  \footnotesize\begin{tabular}{@{}l@{}}#1\end{tabular}%
}

\begin{tikzpicture}
  % nodes
  \node (D) at (3cm,0) {$\mathrm{Ker}H$};
  \node (B) at (0cm,1.5cm) {$\Sha(E/L)[2]^{\mathrm{Gal(L/K)}}$};
  \node (C) at (1.5cm,4cm) {$\Sha(E/L)[2]$};
  % edges
  \path (B) edge node[above left] { $\#\mathrm{tr}(\Sha(E_D/L)[2]$)}(C);
  \path (C) edge (D);
  \path (D) edge (B);
\end{tikzpicture}

\caption{The gap between $\mathrm{Ker}H$ and $\Sha(E/L)[2]$}

\end{figure}

\begin{lemma}\label{2}
 $\dfrac{\#\Sha(E/L)[2]}{\#\mathrm{Ker}H}\ge \#\mathrm{tr}(\Sha(E/L)[2])$.

\end{lemma}
\begin{proof}
Since $\mathrm{Ker}H\subset \Sha(E/L)[2]^{\mathrm{Gal}(L/K)}\subset \Sha(E/L)[2]$ and  \[\dfrac{\Sha(E/L)[2]}{\Sha(E/L)[2]^{\mathrm{Gal}(L/K)}}\cong (\sigma-1)\Sha(E/L)[2]=\mathrm{tr}(\Sha(E/L)[2]),\]
we obtain $\#\Sha(E/L)[2]\ge 
\#\mathrm{tr}(\Sha(E/L)[2])\#\mathrm{Ker}H$.
\end{proof}

By combining Lemma \ref{1} and Lemma \ref{2} and Proposition \ref{trace lemma}, we obtain 

{\small\begin{align*}
\#\Sha(E/L)[2] &\ge \#\mathrm{tr}(\Sha(E/L)[2])\#\mathrm{Ker}H \quad(\text{by}\ \text{Lemma}\  \ref{2}) \\
&\ge \dfrac{\#\Sha(E_D/K)[2]}{4^{\mathrm{rank}(E/K)}\#E(K)[2]^2\#2 \Sha(E/L)[4]} \#
\mathrm{Ker}H \quad (\text{by}\ \text{Proposition}\ \ref{trace lemma})\\&\ge \dfrac{\#\Sha(E_D/K)[2]\#\Sha(E/K)[2] \# \bigoplus_{v\in M_K} H^1(\mathrm{Gal}(L_w/K_v),E(L_w))}{4^{\mathrm{rank}(E/K)}\#E(K)[2]^2\#j(\mathrm{Coker}F)\#2\Sha(E/L)[4]\#H^1(\mathrm{Gal}(L/K),E(L))
}\quad .\\
\end{align*}}

The last inequality follows from lemma \ref{1}.

Because
\begin{align*}
\#j(\mathrm{Coker}F) &\leq \#X = \#\mathrm{Sel}^2(E/K)( \text{by}\ \text{Theorem}\ \ref{Zhe}) \\
                     &= \#\frac{E(K)}{2E(K)} \times \#\Sha(E/K)[2](\text{by}\ \text{exact sequence} \ \ref{basic}) \\
                     &= \#E(K)[2] \times 2^{\mathrm{rank}(E/K)} \times \#\Sha(E/K)[2]
\end{align*} holds, we obtain the following inequality.

\begin{proposition}\label{main inequality}
Let $E/K$ be an elliptic curve and let $L/K$ be a quadratic field extension of a number field. Then, 
\[\dfrac{\#\Sha(E/L)[4]}{\#\Sha(E_D/K)[2]} \geq \frac{1}{\#E(K)[2]^3 \times 2^{3\text{rank}(E/K)}} \cdot \frac{\# \bigoplus_{v \in M_K} H^1(\mathrm{Gal}(L_w/K_v),E(L_w))}{\#H^1(\mathrm{Gal}(L/K),E(L))}\]
holds. 
\end{proposition}

\begin{remark}
In this paper, we restrict our consideration to the case of the 2-torsion subgroup of the Tate--Shafarevich group. This focus is motivated by the fact that $\#\Sha(E/\Bbb{Q}(\sqrt{D}))[n]=\#\Sha(E_D/\Bbb{Q})[n]\#\Sha(E/\Bbb{Q})[n]$ where $n$ is an odd number. This is because $\mathrm{Gal}(\Bbb{Q}(\sqrt{D})/\Bbb{Q})$ acts on the odd abelian group $\Sha(E/\Bbb{Q}(\sqrt{D}))[n]$, and thus $\Sha(E/\Bbb{Q}(\sqrt{D}))[n]$ decomposes into a direct sum of $\Sha(E/\Bbb{Q}(\sqrt{D}))[n]^{+}\coloneqq \{a\in \Sha(E/\Bbb{Q}(\sqrt{D}))[n]\mid \sigma * a=a\}=\Sha(E/\Bbb{Q})[n]$ and $\Sha(E/L)[n]^{-}\coloneqq \{a\in \Sha(E/\Bbb{Q}(\sqrt{D}))[n]\mid \sigma *a=-a\}$. 
For the definition of the action denoted by $*$, see Definition \ref{def1}. The isomorphism
$\Sha(E/\Bbb{Q}(\sqrt{D}))[n]^{-}\cong \Sha(E_D/\Bbb{Q})[n]$ follows from the commutative diagram in the proof of Proposition \ref{trace lemma} and the isomorphism $\Sha(E_D/\Bbb{Q}(\sqrt{D}))[n]^{\mathrm{Gal}(\Bbb{Q}(\sqrt{D})/\Bbb{Q})} \cong \Sha(E_D/\Bbb{Q})[n]$.

\end{remark}

\begin{remark}\label{difficult1}
Let $K$ be a number field and $E/K$ be an elliptic curve over $K$. Let $L=K(\sqrt{D})$ be a quadratic extension of $K$. H. Yu explicitly expressed the formula for $\#\Sha(E/L)$ under the assumption that the Tate--Shafarevich group of elliptic curves over $K$ are finite, that is, 
\[\dfrac{\#\Sha(E/L)}{\#\Sha(E_D/K)} =\frac{\#\Sha(E/K)}{\#\mathrm{Coker}(\mathrm{trace} : E(L)\to E(K))} \cdot \frac{\# \bigoplus_{v \in M_K} H^1(\mathrm{Gal}(L_w/K_v),E(L_w))}{\#H^1(\mathrm{Gal}(L/K),E(L))}\](see [\cite{Yu}, Main Theorem]).\par

Combining Yu's formula with Remark \ref{localglobal}, we know that \[\#\Sha(E/\Bbb{Q}(\sqrt{D}))/\#\Sha(E_D/\Bbb{Q})\] is unbounded from above under the assumption that Tate--Shafarevich group of elliptic curves over $\Bbb{Q}$ are finite. 

Compare the following diagram from Yu's paper with Figure \ref{figure1}.

%% local command for shorter input
\newcommand{\bop}[1]{%
  \mathop{\smash[b]{\bigoplus\limits_{#1}}}%
}
\hspace*{-9mm}
\begin{tikzcd}[column sep=2em]
&[-0.5em] 0 \arrow[d] & 0 \arrow[d] & 0 \arrow[d] &[-1.5em]
\\
& \mathrm{Ker} F'\arrow[r] \arrow[d]
& \Sha(E/K) \arrow[r] \arrow[d]
& \mathrm{Ker} H' \arrow[d]
\\
& H^1(\mathrm{Gal}(L/K),E(L)) \arrow[r,"\text{inf}"] \arrow[d, "F'"]
& H^1(G_K,E) \arrow[r,"\text{res}"] \arrow[d,"G'"]
& \mathrm{res}(H^1(G_K,E)) \arrow[d,"H'"] \arrow[r]
& 0
\\
0 \arrow[r]
& \bop{v\in M_K} H^1(\mathrm{Gal}(L_w/K_v),E(L_w))
    \arrow[r] \arrow[d]
& \bop{v\in M_K} H^1(G_{K_v},E)
    \arrow[r] \arrow[d]
& \bigoplus\limits_{v\in M_K} \text{res}(H^1(G_{K_v},E))
\\
& \mathrm{Coker} F' \arrow[r,"j'"] \arrow[d]
&{{\widehat{E(K)}}}^* \arrow[d]
\\
& 0 & 0
\end{tikzcd}

Note that the exactness of the middle column sequence requires the  assumption that $\Sha(E/K)$ is finite. In this paper, $X$ plays the role that ${{\widehat{E(K)}}}^*$ (Pontryagin dual of profinite completion of $E(K)$) plays in the diagram above.

\end{remark}

\begin{theorem}\label{cor}
 For an arbitrary $r\in \Bbb{Z}$ and an elliptic curve over $\Bbb{Q}$ with $E(\Bbb{Q})[2]\cong \Bbb{Z}/2\Bbb{Z}$ that does not have a cyclic 4-isogeny defined over $\Bbb{Q}(E[2])$, there exists a square-free integer $D$ such that
 
 $\dfrac{\#\Sha(E/\Bbb{Q}(\sqrt{D}))[4]}{\#\Sha(E_D/\Bbb{Q})[2]}\ge r$ and $\#\Sha(E_D/\Bbb{Q})[2]\ge r$.
\end{theorem}

\begin{proof}
$E/\Bbb{Q}$ has a Weierstrass form $E:y^2=x^3+ax^2+bx$($a,b \in \Bbb{Z}$) and let $E'$ be $E' : y^2=x^3-2ax^2+(a^2-4b)x$. Let $\phi:E\to E'$ be $(x,y)\mapsto \left(\frac{y^2}{x^2}, \frac{y(b-x^2)}{x^2}\right)$ be degree $2$ isogeny and $\hat{\phi}$ be its dual.

The following inequality holds because there exists an exact sequence: $0\to E'(\Bbb{Q})[\hat{\phi}]/\phi(E(\Bbb{Q})[2])\to \mathrm{Sel}^{\phi}(E/\Bbb{Q})\to \mathrm{Sel}^2(E/\Bbb{Q})$ (see \cite{Sch}, lemma 9.1).

\[\#\Sha(E_D/\Bbb{Q})[2]=\frac{\#\mathrm{Sel}^2(E_D/\Bbb{Q})}{\#E_D(\Bbb{Q})/2E_D(\Bbb{Q})}\ge  \frac{\#\mathrm{Sel}^{\phi}(E_D/\Bbb{Q})}{\#\mathrm{Sel}^{\hat{\phi}}({E'_D}/\Bbb{Q})}\times\frac{1}{2\times 2^{\mathrm{rank}(E_D/\Bbb{Q})}\times \#E(\Bbb{Q})[2] }.\]
Here, $\dfrac{\#\mathrm{Sel}^{\phi}(E_D/\Bbb{Q})}{\#\mathrm{Sel}^{\hat{\phi}}({E'_D}/\Bbb{Q})}$ is  what we call the Tamagawa ratio. Let $\Delta$ and $\Delta'$ be discriminants of any integral model of $E$ and $E'$ respectively. By Proposition 3.3 of \cite{Kl},

\[\dfrac{\#\mathrm{Sel}^{\phi}(E_D/\Bbb{Q})}{\#\mathrm{Sel}^{\hat{\phi}}({E'_D}/\Bbb{Q})}\ge \prod_{p\mid D\ \text{and} \ p \nmid \Delta} 2^{\frac{\big(\frac{\Delta'}{p}\big)-\big(\frac{\Delta}{p}\big)}{2}}\] holds. Let $h(D)\coloneqq \prod_{p\mid D\ \text{and} \ p \nmid \Delta} 2^{\frac{\big(\frac{\Delta'}{p}\big)-\big(\frac{\Delta}{p}\big)}{2}}$. It is sufficient to prove
that $\forall{r}\in \Bbb{Z}, \exists D$: square free such that $g(D)\ge r$\ (see Proposition \ref{g(D)} for the definition of $g(D)$) and $h(D)\ge r$. 

For arbitrary $r\in \Bbb{Z}_{\ge1}$, let us take $R$ such that $2^{R-4}\ge r$. By the condition that $E(\Bbb{Q})[2]\cong \Bbb{Z}/2\Bbb{Z}$ and $E$ does not have a cyclic 4-isogeny defined over $\Bbb{Q}(E[2])$, by Proposition 5.2 of \cite{Kl},
we have $\Delta \Delta' \notin {\Bbb{Q}}^{\times 2}$. 
By the Chebotarev density theorem, there exist infinitely many primes $p$ that satisfy the following conditions:  
\begin{enumerate}

\item $\big(\frac{\Delta'}{p}\big)=1$

\item $\big(\frac{\Delta}{p}\big)=-1$

\end{enumerate}
Take $p_1,\cdots, p_R$ that satisfy conditions $(1)$ and $(2)$. Take different odd prime numbers $l_1,\cdots,l_{R}$ that satisfy the conditions $\tilde{E}(\Bbb{F}_{l_i})[2]\cong (\Bbb{Z}/2\Bbb{Z})^2$ and $E/\Bbb{Q}$ has good reduction at $l_i$ (We choose primes $(l_1,\dots,l_R)$ satisfying conditions (1) and (2) in the proof of Proposition \ref{g(D)}). There exists a square-free integer $D_R$ with the number of prime factors no greater than $4$ such that $\mathrm{rank}(E_D/\Bbb{Q})=0$, where $D=p_1\cdots p_Rl_1\cdots l_{R}D_R$ by Theorem \ref{H}. For this $D$, $h(D)\ge r$ holds. Also, \(g(D)\ge r\) holds. Indeed, the number of odd primes $q$ that ramify in the quadratic extension $\mathbb{Q}(\sqrt{D})/\mathbb{Q}$ and are primes of good reduction for $E/\mathbb{Q}$ with $\tilde{E}(\mathbb{F}_q)\cong (\mathbb{Z}/2\mathbb{Z})^2$ is at least $R-4$. Hence $\#\bigoplus_{p} H^1\!\left(\mathrm{Gal}\!\left(\Bbb{Q}_p(\sqrt{D})/\Bbb{Q}_p\right),\,
E\!\left(\Bbb{Q}_p(\sqrt{D})\right)\right) \;\ge\; 4^{R-4}$ by Proposition \ref{local} (3). 
Therefore,
\[
g(D) \;\ge\; \frac{4^{R-4}}{2^{\mathrm{rank}(E_D/\Bbb{Q})}\,\#E(\Bbb{Q})[2]}
\;\ge\; 2^{2R-9} \;\ge\; r.\qedhere
\]
\end{proof}

\section{Decreasing $\Sha(E/\mathbb{Q}(\sqrt{D}))[2]$ and $\Sha(E_D/\mathbb{Q})[2]$ simultaneously}

 In this section, we prove that for a prime $p$, there exist infinitely many quadratic fields $K=\Bbb{Q}(\sqrt{D})$ for the elliptic curve $E : y^2=x^3+px$, such that $\Sha(E_D/\Bbb{Q})[2]=0$ and $\#\Sha(E/K)[2]\le 4$ respectively under the assumption that Tate--Shafarevich group is finite. 
 
 Let us recall the theory of the descent using the two 2-isogenies $\phi$ and $\hat{\phi}$ as described in \cite{sil}. Let $E/K$ be an elliptic defined by $E : y^2=x^3+ax\ (a \in \Bbb{Z})$. Let $E': y^2=x^3-4ax$. Let $\phi:E\to E'$ be $(x,y)\mapsto \left(\frac{y^2}{x^2}, \frac{y(a-x^2)}{x^2}\right)$ be degree $2$ isogeny and $\hat{\phi}$ be its dual. Let \[S_{E/K}\coloneqq \{v\in M_K : E\ \text{has bad reduction at}\ v \} \cup \{\text{infinite places of}\ K\}.\]
When we fix $E/\mathbb{Q}$, $\#S_{E/K}$ depends on $K$. Let $H^1(G_K,E[\phi]; S)$ be \[H^1(G_K,E[\phi]; S) \stackrel{\mathrm{def}}{=}\{[\sigma] \in H^1(G_K,E[\phi])\mid \sigma \ \text{is unramified outside}\  S\}\] where unramified outside $S$ means restriction of $\sigma$ to inertia group $I_v\coloneq \mathrm{Gal}(\overline{{{K_v}^{nr}}}/{K_v}^{\mathrm{nr}})$ \ at \ $v\notin S$ is trivial. The Selmer group $K(S,2)$ of a field $K$, which is a finite group, is defined as \[K(S,2)\stackrel{\mathrm{def}}{=} \{\bar{d}\in K^{\times}/{K^{\times}}^2 \mid v(d)\equiv 0\bmod 2, \forall v\notin S  \}\]

The Selmer group $\mathrm{Sel}^{\phi}(E/K)$ is embedded into $K(S,2)$ via \[\mathrm{Sel}^{\phi}(E/K)\subset H^1(G_K,E[\phi];S)\cong K(S,2)\] where the last isomorphism $K(S,2)\cong H^1(G_K,E[\phi]; S)$ is given by $\overline{d} \mapsto [f_d : \sigma \mapsto \frac{{\sqrt{d}}^{\sigma}}{\sqrt{d}}]$.  Here, $(-)^{\sigma}$ denotes the action of $\sigma \in G_K$ on elements of $\overline{K}$. Note that $f_d$ is unramified outside $S$ if only if $v(d)\equiv 0 \mod 2$ for $v\notin S$. Here, note that we identify $E[\phi] \cong \mu_2$ as a trivial $G_K$-module. \par

Let $C_d$ be the image of $\overline{d}$ by the composition $K(S,2)\cong H^1(G_K,E[\phi];S)\subset H^1(G_K,E[\phi])\twoheadrightarrow WC(E/K)[\phi]$. [\cite{sil}, Proposition $4.9$] shows that $C_d$ is isomorphic over $K$ to a projective closure of $dy^2=d^2-4ax^4$ in $\Bbb{P}_K(1,2,1)$. The Selmer group is cut out from $K(S,2)$  by the condition that corresponding torsors has a rational points locally at bad primes in $S$. Let $C_d$ and $C'_d$ be the curves defined by the equations $dy^2 = d^2 - 4ax^4$ and $dy^2 = d^2 + 16ax^4$, respectively. Then, 

\[
\mathrm{Sel}^{\phi}(E/K)\cong 
\left\{
  \bar{d}\in K(S,2)
  \;\middle|\;
  \begin{aligned}
    & C_d(K_v)\neq \emptyset , \forall v\in S, \\
    & C_d: dy^2=d^2-4ax^4
  \end{aligned}
\right\}.
\]

By replacing $a\mapsto -4a$, we obtain the following.

\[
\mathrm{Sel}^{\hat{\phi}}(E'/K)\cong 
\left\{
  \bar{d}\in K(S,2)
  \;\middle|\;
  \begin{aligned}
    & C'_d(K_v)\neq \emptyset , \forall v\in S, \\
    & C'_d: dy^2=d^2+16ax^4
  \end{aligned}
\right\}.
\]

\begin{remark}\label{infinite}
When we write $C_d : dy^2=d^2-4ax^4$, it precisely represents the projective curve obtained by embedding $dy^2=d^2-4ax^4$ into the weighted projective space $\mathbb{P}(1,2,1)$. Simply taking the projective closure in $\mathbb{P}^2$ would result in a singular point at $[0:1:0]$, which cannot be adopted as a torsor. Therefore, to eliminate the singular point, we glue together two nonsingular affine curves, $C_0: dy^2=d^2-4ax^4$ and $C_1:dv^2=d^2u^4-4a$, using the relation $u=\frac{1}{x}, v=\frac{y}{x^2}$.
Specifically, we embed these curves into $\mathbb{P}(1,2,1)$ as follows: $i: C_0 \to \mathbb{P}(1,2,1)$ given by $(x,y) \mapsto [x:y:1]$ and 
$v: C_1 \to \mathbb{P}(1,2,1)$ given by $(u,v) \mapsto [1:v:u]$. We then define the curve $C_d$ in $\mathbb{P}_K(1,2,1)$ as $C_d=i(C_0) \cup v(C_1)$. The projective closure of affine part of $C_d$ has two points \[[1: \pm\sqrt{\dfrac{-4a}{d}}:0]\] at infinity in $\Bbb{P}_K(1,2,1)$. 

\end{remark}

The following inequality gives an upper bound for the order of $\Sha(E/K)[2]$.

\begin{proposition}\label{two sequence}
 Let $E/K$ be an elliptic curve. Let $\phi : E\to E'$ be an isogeny of degree $2$ and $\hat{\phi} : E'\to E$ be the dual isogeny of $\phi$. The following inequality holds:
  
\begin{align*}
\dim_{\Bbb{F}_2} \Sha(E/K)[2] &\le 
\dim_{\Bbb{F}_2} \operatorname{Sel}^{\phi}(E/K) + \dim_{\Bbb{F}_2} \operatorname{Sel}^{\hat{\phi}}(E'/K) \\
&\quad - \dim_{\Bbb{F}_2} \dfrac{E'(K)[\hat{\phi}]}{\phi(E(K)[2])} - \dim_{\Bbb{F}_2} \dfrac{E(K)}{2E(K)}.
\end{align*}

\end{proposition}
\begin{proof}

 There exists an exact sequence: \[0\to E'(K)[\hat{\phi}]/\phi(E(K)[2])\to \mathrm{Sel}^{\phi}(E/K)\to \mathrm{Sel}^2(E/K)\stackrel{\phi}{\to} \mathrm{Sel}^{\hat{\phi}}(E'/K)\] (see [\cite{Sch}, lemma 9.1]). Thus, 
\begin{equation}\dim_{\Bbb{F}_2}\mathrm{Sel}^2(E/K)\le \dim_{\Bbb{F}_2}\mathrm{Sel}^{\phi}(E/K)+\dim_{\Bbb{F}_2}\mathrm{Sel}^{\hat{\phi}}(E'/K)-\dim_{\Bbb{F}_2}\dfrac{E'(K)[\hat{\phi}]}{\phi(E(K)[2])}\label{ine}.\end{equation}

Therefore, 

{\small \[
\begin{aligned}
    &\dim_{\Bbb{F}_2}\Sha(E/K)[2]   \\
    &= \dim_{\Bbb{F}_2}\mathrm{Sel}^2(E/K)-\dim_{\Bbb{F}_2} \frac{E(K)}{2E(K)} \ ( \text{by exact sequence}\ \eqref{basic}) \\
    &\le \dim_{\Bbb{F}_2}\mathrm{Sel}^{\phi}(E/K)+\dim_{\Bbb{F}_2}\mathrm{Sel}^{\hat{\phi}}(E'/K)-\dim_{\Bbb{F}_2} \frac{E(K)}{2E(K)} -\dim_{\Bbb{F}_2}\dfrac{E'(K)[\hat{\phi}]}{\phi(E(K)[2])}\ .
\end{aligned}
\]}

The last inequality follows from \eqref{ine}.\end{proof}

From this point forward, we will limit our discussion to elliptic curves of the form $E: y^2=x^3+px$ where $p$ is a prime number.

\begin{theorem}[\mbox{\rm\textit{cf.}} Genus theory, {\cite[Theorem 8, p.~247]{Sha}}]\label{genus}

Let $K=\Bbb{Q}(\sqrt{D})$ be an imaginary quadratic field and $\mathrm{Cl}_K$ be the ideal class group of $K$. Then $\#{\mathrm{Cl}}_K[2] = 2^{r-1}$ holds where $r$ is the number of prime factors of the discriminant of $K$.

\end{theorem}

\begin{proposition}\label{condition}

Let $p$ be an odd prime, and let $E :y^2=x^3+px$ be an elliptic curve. 

Let $S\coloneqq S_{E/K}$.
 Suppose that an imaginary quadratic field $K=\Bbb{Q}(\sqrt{D})$ satisfies the following conditions:  
$\lvert D \rvert \neq p$ is a prime number such that $D\equiv 5 \bmod 8$ and $p$ does not split in $K=\Bbb{Q}(\sqrt{D})/\Bbb{Q}$. Then, we have the following. 
\begin{itemize}
\item[(1)]$\#K(S,2)=8$.

\item[(2)] Assume that $\Sha(E/K)$ is finite. Then $\#\Sha(E/K)[2]\le 4$.

\end{itemize}

\end{proposition}

\begin{proof}

(1) Define  a group of $S$-units as ${O_{K,S}}^{\times} \stackrel{\mathrm{def}}{=} \{ a\in K \mid v(a)=0, \forall v\notin S \}$ and define the $S$-ideal class group $\mathrm{Cl}(K,S)$ as the ideal class group of $O_{K,S}$. There is the following exact sequence (\cite{Poonen}, Proposition $12.6$):
    \[1 \to {O_{K,S}}^{\times}/{{O_{K,S}}^{\times}}^2 \to K(S,2) \to \text{Cl}(K,S)[2] \to 1 .\]
    To prove $\#K(S,2)=8$, let us prove $\mathrm{Cl}(K,S)=1$ and $\#O_{K,S}^{\times}/{O_{K,S}^{\times}}^2=8$. Because $\lvert D \rvert$ is a prime, $\mathrm{Cl}_K[2]=1$ by Theorem \ref{genus}. Hence,  $\mathrm{Cl}_K$ is an abelian group of odd order. There is a surjection from $\mathrm{Cl}_K$ to $\mathrm{Cl}(K,S)$. Therefore, the order of $\mathrm{Cl}(K,S)$ is odd. Hence, $\mathrm{Cl}(K,S)[2]=1$.

Let us consider ${O_{K,S}}^{\times}/{O_{K,S}^{\times}}^2$.
From Dirichlet's $S$-unit theorem, ${O_{K,S}}^{\times} \cong \mu(K) \times \Bbb{Z}^{\#S-1}$ where $\mu(K)$ is the group of roots of unity. Since $D \equiv 5 \bmod 8$, $2$ does not split in $K$, and $p$ does not split in $K$ by hypothesis, $\# S = 3$. Since $K$ is an imaginary quadratic field, $\mu(K)/\mu(K)^2 \cong \Bbb{Z}/2\Bbb{Z}$.
From the above, ${O_{K,S}}^{\times}/{O_{K,S}^{\times}}^2 \cong (\Bbb{Z}/2\Bbb{Z})^3$. 
Therefore, $\#K(S,2) = \# {O_{K,S}}^{\times}/{O_{K,S}^{\times}}^2 \times \#\mathrm{Cl}(K,S)[2] = 8$ holds.
$\hfill \square$

(2) Let $\mathrm{dim}_{\Bbb{F}_2} \mathrm{Sel}^{\phi}(E/K)=a$ and $\mathrm{dim}_{\Bbb{F}_2} \mathrm{Sel}^{\hat{\phi}}(E'/K)=b$. By $(1)$, $a,b \le 3$. The $\hat{\phi}$-Selmer group is,\[\mathrm{Sel}^{\hat{\phi}}(E'/K)\cong \{\bar{d}\in K(S,2)\mid C'_d(\Bbb{Q}_p)\neq \emptyset, \forall{p}\in \{2,p,\infty\}\}.\]

When $D\equiv 5 \bmod 8$, $C'_2:2y^2=4+px^4$ does not have a root in $\Bbb{Q}_2(\sqrt{D})=\Bbb{Q}_2(\sqrt{5})$. 
Note that $\Bbb{Q}_2(\sqrt{5})=\Bbb{Q}_2(\zeta_3)$ is an unramified extension of $\Bbb{Q}_2$, thus $2$-adic valuation $v_2$ takes integer valuation. 
If $C'_2$ had a $\Bbb{Q}_2(\sqrt{5})$-rational point $(x,y)$, looking at the $2$-adic valuation $v_2$ of both sides, we obtain $1+2v_2(y)=\min\{2,4v_2(x)\}$. But the left-hand side is odd and the right-hand side is even, which is a contradiction. Also, points at infinity of $C'_2$ are $[1:\pm 4\sqrt{\dfrac{p}{2}}:0]=[1:\pm 4\sqrt{2p}:0]$ by  Remark \ref{infinite}. Because $\sqrt{2p}\notin \Bbb{Q}_2(\sqrt{5})$, there are no $\Bbb{Q}_2(\sqrt{5})$-rational points at infinity. Thus, we obtain that $2 \notin \mathrm{Sel}^{\hat{\phi}}(E'/K)$. Note that $\#E(K)[\hat{\phi}]/\phi(E(K)[2])=2$. Thus, $b\le 2$, we can conclude 
$\mathrm{dim}_{\Bbb{F}_2} \Sha(E/K)[2]\le a+b-1-1 \le 3$ by Proposition \ref{two sequence}. Because $\dim_{\Bbb{F}_2}\Sha(E/K)[2]$ is even when $\Sha(E/K)$ is finite (see [\cite{sil}, Section X, Remark 6.3]), $\mathrm{dim}_{\Bbb{F}_2} \Sha(E/K)[2]\le 2$ holds.\end{proof}

\begin{proposition}\label{mainlemma}
Let $p$ be an odd prime, and let $E :y^2=x^3+px$ be an elliptic curve.

\begin{itemize}
\item[(1)]
When $p\equiv 1\bmod 4$, we take an imaginary quadratic field $K=\Bbb{Q}(\sqrt{D})$ that satisfies the following conditions:

\begin{itemize} [label=\textbullet] 
\item $l\coloneqq -D \neq p$ is a prime number,
\item $D\equiv 1 \bmod 4$,
\item $p$ does not split in $K=\Bbb{Q}(\sqrt{D})/\Bbb{Q}$.
\end{itemize}
Assume that $\Sha(E_D/\Bbb{Q})$ is finite. Then, $\Sha(E_D/\Bbb{Q})[2]=0$.

\item[(2)]
When $p\equiv 3\bmod 4$, we take an imaginary quadratic field $K=\Bbb{Q}(\sqrt{D})$ that satisfies the following conditions:
\begin{itemize} [label=\textbullet] 
\item $l\coloneqq -D \neq p$ is a prime number,
\item $D\equiv 3 \bmod 4$,
\item $p$ splits in $K=\Bbb{Q}(\sqrt{D})/\Bbb{Q}$.
\end{itemize}
Assume that $\Sha(E_D/\Bbb{Q})$ is finite. Then, $\Sha(E_D/\Bbb{Q})[2]=0$.

\end{itemize}

\end{proposition}

\begin{proof}

First, let us establish the common preliminary setup for both cases (1) and (2).
Let $S'\coloneqq S_{E_D/\Bbb{Q}}$. Since $-D$ is a prime number, $S'=\{2,p,-D,\infty\}$ and \[\#\Bbb{Q}(S',2)=\#\{(-1)^{n_1}2^{n_2}p^{n_3}(-D)^{n_4}\mid 0\le n_1,n_2,n_3,n_4\le 1\}=16.\] Let $\phi_D: E_D\to E_D', (x,y)\mapsto (\frac{y^2}{x^2}, \frac{y(pD^2-x^2)}{x^2})$ be a degree $2$ isogeny and $\hat{\phi}_D$ be dual isogeny. 
Let $\mathrm{dim}_{\Bbb{F}_2}\mathrm{Sel}^{\phi_D}(E_D/\Bbb{Q})=e, \mathrm{dim}_{\Bbb{F}_2}\mathrm{Sel}^{\hat{\phi}_D}(E'_D/\Bbb{Q})=f$. Since $\#\Bbb{Q}(S',2)=16, e, f\le 4$.

 Let $T_d: dy^2=d^2-4pD^2x^4$ and $T'_d: dy^2=d^2+pD^2x^4$. 
 
The $\phi_D$-Selmer group and $\hat{\phi}_D$-Selmer group are as follows:
\[
\mathrm{Sel}^{\phi_D}(E_D/\Bbb{Q})\cong \left\{ \bar{d} \in \mathbb{Q}(S',2) \mid T_d(\Bbb{Q}_v)\neq \emptyset, \forall{v}\in \{2,p,l,\infty\}\right\}, 
\]
\[
\mathrm{Sel}^{\hat{\phi}_D}(E'_D/\Bbb{Q}) \cong \left\{ \bar{d} \in \mathbb{Q}(S',2) \mid T'_d(\Bbb{Q}_v)\neq \emptyset, \forall{v}\in \{2,p,l,\infty\}\right\}.
\]
For curves $T_d$ and $T'_d$, we first determine their points at infinity: By Remark \ref{infinite}, the points at infinity  of $T_d$ are $[1:\pm 2D\sqrt{\frac{-p}{d}}:0]$. Similarly, the points at infinity of $T'_d$ are $[1:\pm D\sqrt{\frac{p}{d}}: 0]$. 
   
For both (1) and (2), $\bigg(\frac{p}{l}\bigg)=-1$ holds true.

We prove that $T_D: y^2=D-4pDx^4, T_{pD}: y^2=pD-4Dx^4, T_{2pD}: y^2=2pD-2Dx^4,T_{2D}: y^2=2D-2pDx^4$ do not have $\Bbb{Q}_l$-rational points. For each curve, if there exists a $\mathbb{Q}_l$-rational point $(x,y)$, then the Hilbert symbol $(A,B)_l$ of $A(x^2)^2+By^2=1$ must be $1$. We compute the Hilbert symbols mod $l$ for these quadratic forms:
For $T_D: (\frac{1}{D}, 4p)_l = (p,D)_l$.
For $T_{pD}: (\frac{1}{pD}, \frac{4}{p})_l = (pD,p)_l = (p,D)_l$.
For $T_{2pD}: (\frac{1}{2pD}, \frac{1}{p})_l = (2pD,p)_l = (p,D)_l$.
For $T_{2D}: (\frac{1}{2D}, p)_l = (2D,p)_l = (p,D)_l$.
Since $\left( \dfrac{p}{l} \right) = -1$, we have $(p,D)_l = -1$. Therefore, all the above Hilbert symbols equal $-1$, which proves that none of these curves have $\mathbb{Q}_l$-rational points in the affine part. Moreover, there are no $\Bbb{Q}_l$-rational points at infinity because $\sqrt{pl}, \sqrt{l}, \sqrt{2l}, \sqrt{2pl} \notin \mathbb{Q}_l$.
As a consequence, we conclude that $D, pD, 2pD, 2D \notin \mathrm{Sel}^{\phi_D}(E_D/\mathbb{Q})$.

Note that $-p\in \mathrm{Sel}^{\phi_D}(E_D/\Bbb{Q})$ since the points at infinity are $[1:\pm{2D}:0]$. Therefore, $-pD,-D,-2D,-2pD\notin \mathrm{Sel}^{\phi_D}(E_D/\Bbb{Q})$. Since we have determined that $8$ out of $16$ elements of $\Bbb{Q}(S',2)$ do not belong to $\mathrm{Sel}^{\phi_D}(E_D/\Bbb{Q})$, if we can show that one of the remaining $T_d$ has no $\Bbb{Q}_v$-rational point for some $v \in S'$, then we can prove that $e \le 2$.

Let us evaluate $f$. When $d < 0$, $T'_d(\mathbb{R})=\emptyset$. Thus, $d\notin \mathrm{Sel}^{\hat{\phi}_D}(E_D/\Bbb{Q})$. Affine parts of $T'_2: 2y^2=4+pD^2x^4$, $T'_{2pl}: 2y^2=4pl+lx^4$  and $T'_{2p}: 2y^2=4p+D^2x^4$ do not have $\Bbb{Q}_2$-rational points because $1+2\mathrm{ord}_2(y)=\mathrm{min}\{2,4\mathrm{ord}_2(x)\}$ does not hold. Also, note that $T'_2, T'_{2p}, T'_{2l}, T'_{2pl}$ do not have $\Bbb{Q}_2$-rational points at infinity since none of $\sqrt{2p}, \sqrt{2}$,$\sqrt{2pl}$,$\sqrt{2l}$ belongs to $\Bbb{Q}_2$. Since we have shown that $12$ out of $16$ elements do not belong to $\mathrm{Sel}^{\hat{\phi}_D}(E'_D/\Bbb{Q})$, $f \le 2$,  if we can show that one of the remaining $T_d$ has no $\Bbb{Q}_v$-rational point for some $v \in S'$, then we can prove that $f \le 1$.

\vskip\baselineskip
In what follows, we prove that both $e$ and $f$ can be reduced by $1$ in each of cases (1) and (2).

(1) Let us evaluate $e$. We claim that $T_{-1}: y^2=-1+4pl^2x^4$ does not have $\Bbb{Q}_l$-rational points. Indeed, when $x,y\in \Bbb{Z}_l$, $y^2\equiv -1 \bmod l$ does not hold since $\left( \dfrac{-1}{l}\right)=-1$. When $v_l(x)<0$, $v_l(y)=1+2v_l(x)$. Set $v_l(x)\coloneqq -a$ ($a>0$). We can put $x=l^{-a}x'$, $y=l^{1-2a}y'$ where $x',y'\in {\Bbb{Z}_l}^{\times}$. We obtain $y'^2=-l^{4a-2}+4px'^4$ and $p$ should be a square modulo $l$. This contradicts the fact that $\left( \dfrac{p}{l}\right)=-1$. Thus, $v_l(x)\ge 0$. In this case, $x,y \in \Bbb{Z}_l$ and we have already shown that there are no $\Bbb{Q}_l$-rational points in this case. There are no $\Bbb{Q}_l$-rational points at infinity since $\sqrt{p}\notin \Bbb{Q}_l$. Therefore, $-1\notin \mathrm{Sel}^{\phi_D}(E_D/\Bbb{Q})$. 
Since we have shown that $9$ out of $16$ elements do not belong to $\mathrm{Sel}^{\phi_D}(E_D/\Bbb{Q})$, we can conclude that $e\le 2$.

Let us evaluate $f$. We claim that $T'_l(\Bbb{Q}_p)=\emptyset$ where $T'_l: y^2=l+plx^4$. Indeed, When $p\equiv 1 \bmod 4$, $\left(\frac{1}{l},-p\right)_p=(-1)^{\frac{p+1}{2}}=-1$ and there are no $\Bbb{Q}_p$-rational points at infinity since $\sqrt{-pl}\notin \Bbb{Q}_p$. Thus, $l\notin \mathrm{Sel}^{\hat{\phi}_D}(E'_D/\Bbb{Q})$. Since we have shown that $13$ out of $16$ elements do not belong to $\mathrm{Sel}^{\hat{\phi}_D}(E'_D/\Bbb{Q})$, we can conclude that $f\le 1$.
 Thus, $\mathrm{dim}_{\Bbb{F}_2} \Sha(E_D/\Bbb{Q})[2]\le e+f-1-1 \le 2+1-1-1=1$. Since $\dim_{\Bbb{F}_2}\Sha(E_D/\Bbb{Q})[2]$ is even (see [\cite{sil}, Remark 6.3]), $\mathrm{dim}_{\Bbb{F}_2} \Sha(E_D/\Bbb{Q})[2]=0$ holds. 

(2) Let us evaluate $e$. We claim that $T_{-1}(\Bbb{Q}_p)= \emptyset$ where $T_{-1}: y^2=-1+4pl^2x^4$. Indeed, Hilbert symbol $\left(-1, 4pl^2\right)_p= \left(-1, p\right)_p=(-1)^{\frac{p-1}{2}}=-1$ and there is no $\Bbb{Q}_p$-rational points at infinity because $\sqrt{p}\notin \Bbb{Q}_p$. Since we have shown that $9$ out of $16$ elements do not belong to $\mathrm{Sel}^{\phi_D}(E_D/\Bbb{Q})$, we can conclude that $e\le 2$. 

Let us evaluate $f$. Since $l\equiv 1 \bmod 4$, affine part of $T'_l : y^2=l+plx^4$ does not have $\Bbb{Q}_l$-rational points because $\left(-p,\frac{1}{l}\right)_l=\left(-p, l\right)_l=\bigg( \dfrac{p}{l} \bigg) \bigg( \dfrac{-1}{l} \bigg)=-(-1)^{\frac{l-1}{2}}=-1$. Also, note that $T'_l$ does not have $\Bbb{Q}_l$-rational points at infinity because $\sqrt{pl}$ does not belong to $\Bbb{Q}_l$. Since we have shown that $13$ out of $16$ elements do not belong to $\mathrm{Sel}^{\hat{\phi}_D}(E'_D/\Bbb{Q})$, $f \le 1$.
Thus, $\mathrm{dim}_{\Bbb{F}_2} \Sha(E_D/\Bbb{Q})[2] \le e + f - 1 - 1 \le 2 + 1 - 1 - 1 = 1$. Since $\dim_{\Bbb{F}_2}\Sha(E_D/\Bbb{Q})[2]$ is even, $\mathrm{dim}_{\Bbb{F}_2} \Sha(E_D/\Bbb{Q})[2]=0$ holds.\end{proof}

\begin{proposition}\label{main theorem 3}

Let $p\equiv 1 \bmod 4$ be a prime number, and let $E:y^2=x^3+px$ be an elliptic curve.  
\begin{enumerate}
\item 
There exist infinitely many imaginary quadratic fields 
$K=\Bbb{Q}(\sqrt{D})$ with $-D$ being a prime number such that $\#\Sha(E/K)[2] \le 4$ and $\Sha(E_D/\Bbb{Q})[2]=0$ under the assumption that $\#\Sha(E/K)$ and $\#\Sha(E_D/\Bbb{Q})$ are finite. 
\item 

If $\Sha(E/\Bbb{Q})$ contains an element of order $4$, then for any quadratic number field $K=\Bbb{Q}(\sqrt{D})$, $\Sha(E/K)[2] \neq 0$.
 
\end{enumerate}
\end{proposition}

\begin{proof}
 (1) By Dirichlet's prime number theorem and the Chinese remainder theorem, there exist infinitely many prime numbers $-D$ such that $D\equiv 5 \bmod 8$ and $p$ does not split in $K=\Bbb{Q}(\sqrt{D})/\Bbb{Q}$. By Proposition \ref{condition} and Proposition \ref{mainlemma}, the result follows.

 (2) 
 Let $C\in \Sha(E/\Bbb{Q})$ be an element of order $4$. If $C$ is trivial in $\Sha(E/\Bbb{Q}(\sqrt{D}))$ for some $D$, the order (period), which is $4$, would have to divide index, which is $2$. This is a contradiction. Thus any $K$ cannot trivialize $C$ in $\Sha(E/K)$. This implies $\Sha(E/K)$ has an element of order $2$.\end{proof}

\begin{example}\label{impossible}
Let $p = 257$ be the fourth Fermat prime. Using \textsc{Magma} \cite{Magma}, we compute:\par
\texttt{A := EllipticCurve([0,0,0,257,0]);}\par
\texttt{MordellWeilShaInformation(A: ShaInfo := true);}\par
This computation shows that the Tate--Shafarevich group $\Sha(E/\Bbb{Q})$ of the elliptic curve $E: y^2 = x^3 + 257x$ has an element of order $4$. Thus, by Proposition \ref{main theorem 3} (2), there exists no quadratic field $\Bbb{Q}(\sqrt{D})$ such that $\Sha(E/\Bbb{Q}(\sqrt{D}))[2]=0$.
\end{example}

\begin{remark}
It remains unknown whether there exist finite extensions $L/\Bbb{Q}$ such that \[\Sha(E/L)[2] = 0\].
\end{remark}

\section{Acknowledgement}
 I heartily thank my advisor Nobuo Tsuzuki for his constant encouragement and helpful comments and ideas. I would like to express my sincere gratitude for valuable comments from Professor Kazuo Matsuno. I would like to express my sincere gratitude to Professor Christian Wuthrich, who warmly answered my questions about the Tate-Shafarevich group through online forums and email. I thank Kaoru Sano for valuable discussions during my internship at NTT, and the anonymous referee for a careful reading and helpful comments on the exposition. This work was supported by JSTSPRING, Grant Number JPMJSP2114.

\end{document}